\documentclass[12pt,reqno, a4paper]{amsart}
\usepackage[utf8]{inputenc}
\usepackage{graphicx}
\usepackage{fancyhdr}
\usepackage{enumerate}
\usepackage{amsmath}
\usepackage{amsthm}
\usepackage{mathabx}
\usepackage{amssymb}
\usepackage[colorlinks=true,linkcolor=blue,citecolor=magenta]{hyperref}

\newtheorem{theorem}{Theorem}[section]
\newtheorem{lemma}[theorem]{Lemma}
\newtheorem{proposition}[theorem]{Proposition}
\newtheorem{corollary}[theorem]{Corollary}

\newtheorem{remark}[theorem]{Remark}

\newtheorem{assumption}[theorem]{Assumption}

\newcommand{\cal}{\mathcal}
\newcommand{\mc}[1]{{\mathcal #1}}
\newcommand{\mf}[1]{{\mathfrak #1}}
\newcommand{\mb}[1]{{\mathbf #1}}
\newcommand{\bb}[1]{{\mathbb #1}}
\newcommand{\bs}[1]{{\boldsymbol #1}}

\newcommand\C{{\mathbb C}}
\newcommand\E{{\mathbb E}}
\newcommand\N{{\mathbb N}}

\newcommand\R{{\mathbb R}}

\newcommand\T{{\mathbb T}}

\newcommand\Z{{\mathbb Z}}

\newcommand \ga{\gamma}

\newcommand \la{\lambda}
\newcommand \si{\sigma}

\newcommand{\dd  }{\mathrm{d}}
\newcommand{\A}{\mathcal{A}}
\renewcommand{\S}{\mathcal{S}}

\numberwithin{equation}{section}

\begin{document}

\title{Macroscopic evolution of mechanical and thermal energy in a harmonic chain with random flip of velocities}

\author{Tomasz Komorowski}
\address{Tomasz Komorowski\\Institute of Mathematics, Polish Academy
  Of Sciences\\Warsaw, Poland.} 
\email{{\tt komorow@hektor.umcs.lublin.pl}}

\author{Stefano Olla}
  \address{Universit\'e Paris-Dauphine, PSL Research University\\
  CNRS, CEREMADE\\
  75016 Paris, France }
  \email{olla@ceremade.dauphine.fr}

\author{Marielle Simon}

\address{Inria Lille -- Nord Europe \\ 40 avenue du Halley 
\\ 59650 Villeneuve d'Ascq, France\\ \emph{and} Laboratoire Paul Painlev\'e, UMR CNRS 8524 
\\  Cit\'e Scientifique \\ 59655 Villeneuve d'Ascq, France} \email{marielle.simon@inria.fr}





\begin{abstract}
We consider an unpinned chain of harmonic oscillators with periodic boundary conditions, 
whose dynamics is perturbed by a random flip of 
the sign of the velocities. 
The dynamics conserves the total volume (or elongation) and the total energy of the system. 
We prove that in a diffusive space-time scaling limit the profiles corresponding 
to the two conserved quantities converge to the solution of a diffusive system of differential equations. 
While the elongation follows a simple autonomous linear diffusive equation, 
the evolution of the energy depends  on the gradient of the square of the elongation.  

\end{abstract}

\maketitle

\tableofcontents
\section{Introduction}\label{sec:intro}

Harmonic chains with energy conserving random perturbations of the dynamics have recently 
received attention in the study of the macroscopic evolution of energy \cite{bbjko,bo1,jko,ko,lmn,simon}. 
They provide models that have a non-trivial macroscopic behavior which can be explicitly computed. 
We consider here the dynamics of an unpinned chain where the velocities of particles 
can randomly change sign.
This random mechanism is equivalent to the deterministic collisions 
with independent \emph{environment}  particles of infinite mass.
 Since the chain is unpinned, the \emph{relevant} conserved quantities of the
 dynamics are the {\em energy} and the {\em volume} (or elongation).
 
Under a diffusive space-time scaling, we prove that the profile of elongation evolves 
independently of the energy and follows 
the linear diffusive equation
\begin{equation}
  \partial_t r(t,u)  =  \frac{1}{2\gamma} \partial^2_{uu} r(t,u). \label{eq:hydro-elongation}
\end{equation}
Here $u$ is the Lagrangian space coordinate of the system 
and $\gamma > 0$ is the intensity of the random mechanism of collisions.
The energy profile can be decomposed into the sum of  
\emph{mechanical}  and  \emph{thermal} energy
\begin{equation*}
  e(t,u) = e_{\rm{mech}}(t,u) + e_{\rm{thm}}(t,u) 
\end{equation*}
where the mechanical energy is given by $e_{\rm{mech}}(t,u) = \frac 12 r(t,u)^2$, 
while the thermal part $e_{\rm{thm}}(t,u)$, that coincides 
with the temperature profile, evolves following the non-linear equation: 
\begin{equation}
  \partial_t e_{\rm{thm}}(t,u) = \frac 1{4\gamma} \partial^2_{uu} e_{\rm{thm}}(t,u) 
+ \frac{1}{2\gamma} \left(\partial_{u} r(t,u)\right)^2 . 
\label{eq:hydro-temperature}
\end{equation} 
This is equivalent to the following conservation law for the total energy:
\begin{equation}
    \partial_t e (t,u) = \frac 1{4\gamma} \partial^2_{uu}\bigg(e(t,u) +  \frac{r(t,u)^2}{2} \bigg). 
\label{eq:hydro-energy}
\end{equation}
The derivation of the macroscopic equations \eqref{eq:hydro-elongation} 
and \eqref{eq:hydro-temperature} from the microscopic dynamical system of particles, 
after a diffusive rescaling of space and time, is the goal of this paper.
Concerning the distribution of the energy in the frequency modes: 
the mechanical energy $e_{\rm{mech}}(t,u)$ is concentrated on the modes 
corresponding to the largest wavelength, while the thermal energy $e_{\rm{thm}}(t,u)$ 
is distributed uniformly over  all frequencies. Note that 
$ \frac{1}{2\gamma} \left(\partial_{u} r(t,u)\right)^2$ is the rate of 
dissipation of the mechanical energy into thermal energy.

The presence of the non-linearity in the evolution of the energy makes the macroscopic limit non-trivial. 
Relative entropy methods (as introduced in \cite{yau}) identify correctly the limit equation
 (see \cite{simon}), but in order to make them rigorous one needs sharp bounds 
on higher moments than cannot be controlled by the 
relative entropy\footnote{For more 
details related to these moments bounds, that are still conjectured but not proved 
(on the contrary to what is claimed in \cite{simon}), 
we refer the reader to an \textit{erratum} 
which is available online at 
\href{http://chercheurs.lille.inria.fr/masimon/erratum-v2.pdf}
{http://chercheurs.lille.inria.fr/masimon/erratum-v2.pdf}.}. 
In this sense the proof in \cite{simon} is not complete.

We follow here a  different approach based on Wigner distributions.
The Wigner distributions permit to control the energy distribution over
various frequency modes and provide a natural separation between  mechanical and thermal energies. 
The initial positions and velocities of particles can be random, and the only condition we ask, 
besides to have definite mean asymptotic profiles
of elongation and energy, is that the thermal energy spectrum has a square integrable density. 
In the macroscopic limit we prove that locally the thermal energy spectrum
 has a constant density equal to the local thermal energy (or temperature), i.e.~that the system is,
 at macroscopic positive times, in local equilibrium, even though it is not at initial time.
Also follows from our result that the mechanical energy is concentrated on the lowest modes. 
This is a stronger local equilibrium result than the one usually obtained with relative entropy techniques.
The Wigner distribution approach had been successfully applied in different contexts
for systems perturbed by noise with more conservation laws in \cite{jko,ko}. 
Here we need a particular \emph{asymmetric} version of the Wigner distribution, in order to deal
 with a finite size discrete microscopic system.

When the system is pinned, only energy is conserved and its macroscopic evolution is linear, 
and much easier to be obtained. In this case the thermalization and the correlation structure 
have been studied in \cite{lukk,lmn}.

When the chain of oscillators is \emph{anharmonic}, still with velocity flip dynamics, the hydrodynamic limit 
is a difficult non-gradient problem,
for the moment still open. 
In that case the macroscopic equations  would be:
\begin{equation}
  \label{eq:6}
  \begin{split}
     \partial_t r(t,u)  &=  \frac{1}{2\gamma} \partial^2_{uu} \big[\tau(r, e)\big],\\
      \partial_t e (t,u) &= \partial_u \big[\mathcal D(r,e) \partial_u  \beta^{-1}(r,e) \big] 
      + \frac 1{4\gamma}\partial^2_{uu}\left(\tau(r,e)^2\right).
  \end{split}
\end{equation}
where $\tau(r,e)$ is the thermodynamic equilibrium tension as function of the volume $r$
and of the energy $e$, and $\beta^{-1}(r,e)$ is the corresponding temperature,  
while $\mathcal D(r,e)$ is the thermal diffusivity defined by the usual Green-Kubo formula, 
as space-time variance of the energy current in the equilibrium infinite dynamics at average elongation 
$r$ and energy $e$ (see Section \ref{sec:conj-anharm-inter} for the definition of these quantities).
The linear response and the existence of $\mathcal D(r,e)$ have been proven in \cite{bo2}.



\section{Microscopic dynamics}
\label{sec:adiab-micr-dynam}

\subsection{Periodic chain of oscillators}

 In the following we denote by $\T_n:=\Z/n\Z=\{0,\ldots,n-1\}$ the discrete circle with $n$ points, and, for any $L>0$, by $\T(L)$ thecontinuous circle of length $L$, and we set $\T := \T(1)$. 

We consider a one-dimensional harmonic chain of $n$ oscillators, all of mass 1, 
with periodic boundary conditions.
The  clearest way to describe this system is as a massive one
dimensional discrete surface $\{ \varphi_x \in \R, x\in \T_n \}$.
The element (or particle) $x$ of the surface is at height $\varphi_x$ and has mass
equal to 1. We call  its velocity  (that coincides with its
momentum) $p_x\in \R$. 
Each particle $x$ is connected to the particles $x-1$ and $x+1$ by
harmonic springs, so that $n-1$ and $0$ are connected in the same way. 
The total energy of the system is given by the Hamiltonian:
  \begin{equation}
{\mathcal{H}_{n}}:=\sum_{x\in\T_n} \mathcal{E}_x, \qquad \mathcal{E}_x:=
\frac{p_x^2}2 + \frac{(\varphi_x-\varphi_{x-1})^2 }2 .\label{eq:ham}
\end{equation}
In addition to the Hamiltonian dynamics associated to the harmonic potentials, 
particles are subject to a random interaction with the environment: 
at independently distributed random Poissonian times, the momentum $p_x$ 
is flipped into $-p_x$. The resulting equations of the motion are
\begin{equation} \left\{ \begin{aligned}
    \dd   \varphi_x(t) &= n^2 p_x(t)\; \dd   t ,\\
    \dd   p_x(t) &= n^2 \big(\varphi_{x+1}(t) + \varphi_{x-1}(t) - 2 \varphi_x(t)\big)\; \dd   t - 2p_x(t^-)
    \; \dd   \mathcal N_x(\gamma n^2t),
  \end{aligned}\right.\label{eq:dynamics0}\end{equation} for any $x \in \T_n$. 
Here $\{\mathcal N_x(t)\; ; \; t \geq 0,\; x\in\T_n\}$ are $n$ independent Poisson processes of
intensity 1, and the constant $\gamma$ is positive. We have 
already accelerated the time scale by $n^2$, according to the diffusive scaling. 
Notice that the energy $\mathcal{H}_n$ is conserved by this dynamics. 
There is another important conservation law that is given by the sum of the elongations of the springs,
 that we define as follows.
We call $r_x = \varphi_x - \varphi_{x-1}$ the elongation of the spring between $x$ and $x-1$, 
and since $x\in \T_n$ we have $r_0 =  \varphi_0 - \varphi_{n-1}$. 
The equation of the dynamics in these coordinates are given by:
 \begin{equation} \left\{ \begin{aligned}
    \dd   r_x(t) &= n^2 \big(p_x(t) - p_{x-1}(t)\big)\; \dd   t\\
    \dd   p_x(t) &= n^2 \big(r_{x+1}(t) - r_x(t)\big)\; \dd   t - 2p_x(t^-)
    \; \dd   \mathcal N_x(\gamma n^2t), \qquad x\in\T_n.
  \end{aligned}\right.
\label{eq:dynamics}
\end{equation}
This implies that the dynamics is completely defined giving the initial conditions 
 $\{r_x(0), p_x(0), x \in \T_n\}$.

The periodicity in the $\varphi_x$ variables would impose that $\sum_{x=0}^{n-1} r_x(0) = 0$.
On the other hand the dynamics defined by \eqref{eq:dynamics} is well defined also 
if $\sum_{x=0}^{n-1} r_x(0) \neq 0$
and has the conservation law $\sum_{x=0}^{n-1} r_x(t) = \sum_{x=0}^{n-1} r_x(0) := R_n$.  Note that $R_n$ can also assume negative values.
In this case we can picture the particles as $n$ points $q_0, \dots, q_{n-1}\in \T(|R_n|)$, 
the circle of length $|R_n|$.  
These points can be defined as $q_x:= \big[\sum_{y=0}^x r_y\big]_{\text{mod}|R_n|}$, for $x=0, \dots, n-1$.
It follows that $q_n = q_0$. 
We will not use  neither the {$q_x$ coordinates}
nor the {$\varphi_x$ coordinates}, 
but we  consider only the evolution defined by  \eqref{eq:dynamics} 
with initial configurations $\sum_{x=0}^{n-1} r_x(0) =R_n \in \R$.

\subsection{Generator and invariant measures}

The generator of the stochastic dynamics ($\mathbf{r}(t):=\{r_x(t)\}_{x \in \T_n}$, $\mathbf{p}(t):=\{p_x(t)\}_{x\in\T_n}$), is given by
\begin{equation*}
\mathcal L_n:= n^2 \A_n + n^2 \gamma\; \S_n,
\end{equation*}
where the Liouville operator  $\A_n$ is formally given by
\[ \A_n
 =\sum_{x\in\T_n}\bigg\{(p_{x}-p_{x-1})
 \frac{\partial}{\partial  r_x}+(r_{x+1}-r_{x})
 \frac{\partial}{\partial p_x}\bigg\},
 \]
while, for $f : \Omega_n\to \R$,
\begin{equation*}
  \S_n f ({\mb r},{\mb p}) = \sum_{x\in\T_n}\big\{ f({\mb r}, {\mb p}^x) -
    f({\mb r},{\mb p})\big\} 
\end{equation*}
where ${\mb p}^x$ is the configuration that is obtained from $\mb p$ by reversing the sign of the velocity at site $x$, namely: $({\mb p}^x)_y = p_y$ if $y\neq x$ and  $({\mb p}^x)_x = -p_x$.

The two conserved quantities 
\(\mc H_n= \sum_{x\in\T_n}\mc E_x\) and \(R_n=\sum_{x\in\T_n} r_x,\) are
 determined by the initial data (eventually random), 
and typically they should be proportional to $n$: $\mc H_n = ne, R_n = nr$, with $e\in\R_+$ the average energy per particle, and $r\in \R$ the average spring elongation. 
Consequently the system has a two parameters family of
stationary measures given by the canonical Gibbs distributions
\begin{equation*}
   \mu^n_{\tau,\beta}(\dd \mb r,\dd \mb p) = \prod_{x\in\T_n} \exp\big(- \beta (\mc E_x - \tau
      r_x) - \mc G_{\tau,\beta}\big)\; \dd   r_x \dd   p_x, \qquad  \beta>0, \tau \in\R,
\end{equation*}
where
\begin{equation*}
  \mc G_{\tau,\beta} = \log \Big[\sqrt{2\pi \beta^{-1}}\int_\R e^{-\frac{\beta}{2}
      (r^2 - 2\tau r)}\;\dd   r  \Big]= \log \Big[2\pi\beta^{-1} \exp\Big(\frac{\tau^2\beta}{2}\Big)\Big].
\end{equation*}
As usual, the parameters $\beta^{-1}>0$ and $\tau\in\R$ are called respectively  \emph{temperature} and \emph{tension}.
Observe that the function 
\begin{equation}\label{eq:thermor}
r(\tau,\beta) = \beta^{-1} \; \partial_\tau \mc G_{\tau,\beta}=\tau
\end{equation} gives the average equilibrium length in function of
the tension $\tau$, and  
\begin{equation}\label{eq:thermoe}
\mc E(\tau,\beta) 
= \tau \; r(\tau,\beta)  - \partial_\beta \mc G_{\tau,\beta}=\beta^{-1} + \frac{\tau^2}{2}\end{equation} 
is the corresponding thermodynamic internal energy function. Note that the energy $\mc E(\tau,\beta)$ is 
composed by a \emph{thermal} energy $\beta^{-1}$ and a \textit{mechanical} energy $ \frac{\tau^2}{2}$.

\subsection{Hydrodynamic limits}
Let $\mu_n(\dd   \mb r,\dd   \mb p)$ be an  initial  Borel
  probability distribution on $\Omega_n$.
 We denote by $\bb P_n$ the law of the process $\{({\bf r}(t),{\bf p}(t))\; ; \; t \geq 0\}$  starting from the measure $\mu_n$ and generating by $\mathcal{L}_n$, and by  $\bb E_n$ its corresponding expectation.
We are given initial continuous profiles of tension $\{\tau_0(u)\; ; \; u \in \T\}$ and of temperature 
$\{\beta^{-1}_0(u)>0\; ;\; u \in \T\}$. The thermodynamic relations \eqref{eq:thermor} and \eqref{eq:thermoe} 
give the corresponding initial profiles of elongation and energy as
\begin{equation} \label{eq:initial-profile}
r_0(u):=\tau_0(u) \quad \text{ and } \quad e_0(u):=\frac{1}{\beta_0(u)}+\frac{\tau^2_0(u)}{2}, \qquad u\in\T.
\end{equation}
The initial distributions $\mu_n$ are assumed to satisfy  the following mean
convergence statements: 
\begin{align}
    \frac 1n\sum_{x\in\T_n} G\Big(\frac x n\Big)\E_n \big[ r_x(0) \big] & \xrightarrow[n\to+\infty]{} \int_\T G(u)\;
    r_0(u)\; \dd   u, \label{eq:init-elong} \\
 \frac 1n\sum_{x\in\T_n} G\Big(\frac x n\Big) \E_n\big[\mc E_x(0)\big] &\xrightarrow[n\to+\infty]{} \int_\T G(u)\;
  e_0(u)\; \dd   u,  \label{eq:init-ener}
  \end{align}
for any test function $G$ that belongs to the set ${\cal C}^\infty(\T)$ of smooth
functions on the torus. 
We expect the same convergence to happen at the macroscopic time
$t$:
\begin{equation}\label{eq:HL}
\begin{aligned} 
    &\frac 1n\sum_{x\in\T_n} G\Big(\frac x n\Big) \E_n\big[r_x(t)\big] \xrightarrow[n\to+\infty]{} \int_\T G(u)\;
     r(t,u) \; \dd  u,\\
 &\frac 1n\sum_{x\in\T_n} G\Big(\frac x n\Big) \E_n\big[\mc E_x(t)\big] \xrightarrow[n\to+\infty]{} \int_\T G(u)\;
     e(t,u)\; \dd   u,
\end{aligned}
\end{equation}
where the macroscopic evolution for the volume and energy profiles follows the system of equations:
\begin{equation}\label{eq:linear}
\left\{\begin{aligned}              
\partial_t r(t,u)&=\frac{1}{2\gamma}\partial^2_{uu}r(t,u),\\
\partial_t e(t,u)&= 
\frac{1}{4\gamma}\partial^2_{uu}\Big(e+\frac{r^2}{2}\Big)(t,u), \qquad   (t,u)\in \bb R_+ \times \T,
\end{aligned}\right.
\end{equation}
with the initial condition
\begin{equation*}
 r(0,u)= r_0(u),\qquad      
e(0,u)=e_0(u).
\end{equation*}
 The solutions $e(t,\cdot)$, $r(t,\cdot)$ of \eqref{eq:linear} are smooth when $t >0$ 
(the system of partial differential equations is parabolic). 
Note that the evolution of $r(t,u)$ is autonomous of $e(t,u)$.
The precise assumptions that are needed for the
convergence \eqref{eq:HL} are stated in Theorems \ref{theo:hydro} and
\ref{theo:hydro1} below.

\section{Main results}
\label{mainresults}

\subsection{Notations}
\subsubsection{Discrete Fourier transform}
Let us denote by $\widehat{f}$ the Fourier transform of a finite sequence $\{f_x\}_{x\in\T_n}$ of numbers in $\C$,
 defined as follows: 
\begin{equation}\label{eq:fourier}
\widehat{f}(k)=\sum_{x\in\T_n} f_x\; e^{-2i\pi x k},\qquad k\in \widehat{\T}_n:=\big\{0,\tfrac1n,...,\tfrac{n-1}n\big\}.
\end{equation} 
Reciprocally, for any $f: \widehat{\T}_n \to \C$, we denote by $\big\{\widecheck{f}_x\big\}_{x\in\T_n}$ its inverse Fourier transform given by
\begin{equation}
\label{eq:inversefouri}
\widecheck{f}_x = \frac1n\sum_{k \in  \widehat{\T}_n} e^{2i\pi xk} f(k),\qquad x \in \T_n.
\end{equation}
The Parseval identity reads 
\begin{equation}\label{eq:parseval}
\|f\|_{\mb L^2}^2:=\frac{1}{n}\sum_{k\in\widehat{\T}_n} \big|\widehat{f}(k)\big|^2=\sum_{x\in\T_n} \big|f_x\big|^2.
\end{equation}
 If $\{f_x\}_{x\in\T_n}$ and $\{g_x\}_{x\in\T_n}$ are two  sequences indexed by the discrete torus, their convolution is given by
\[ (f\ast g)_x:= \sum_{y \in \T_n} f_y\; g_{x-y},\qquad x \in \T_n.
\]

\subsubsection{Continuous Fourier transform}

{Let $\mc C(\T)$ be the the space of continuous, complex valued  functions on $\T$.}
For any   function $G\in\mc C(\T)$, let $\mc FG:\Z \to \C$  denote its  Fourier transform  given as follows:
\begin{equation}
\mc FG(\eta):= \int_\T G(u) \; e^{-2i\pi u \eta} \; \dd u,\qquad \eta\in\Z.
\label{eq:fourierZ}
\end{equation}
Similar identities to \eqref{eq:inversefouri} and \eqref{eq:parseval} can easily be written: for instance, 
we shall repeatedly use the following
\begin{equation} \label{eq:fourierZZ}
G(u)=\sum_{\eta\in\Z}\mc F G(\eta)\;  e^{2i\pi \eta u}, \qquad u\in\T.
\end{equation}
Note that when $G$ is smooth the Fourier coefficients satisfy
\begin{equation} \label{eq:decay}
\sup_{\eta\in\Z}\Big\{(1+\eta^2)^p|\mc FG(\eta)|\Big\}<+\infty, \qquad \text{ for any } p\in\N. 
\end{equation}
If $J:\T\times\T\to \C$ is  defined on a two-dimensional torus,   we still denote by $\mc F J(\eta,v)$, $(\eta,v)\in\Z\times\T$,  its Fourier transform with respect to the first variable. We equip the set $\mc C^\infty(\T\times\T)$ of smooth (with respect to the first variable)  functions with the norm
\begin{equation}
\label{022104}
\|J\|_{0}:= \sum_{\eta\in\Z}\sup_{v\in\T} \big|\mc FJ(\eta,v)\big|.
\end{equation}
Let  ${\cal A}_0$ be the completion of $\mc C^\infty(\T\times\T)$ in
this norm and  $({\cal A}_0',\|\cdot\|_0')$ its dual space. 

\subsubsection{A fundamental example} 

In what follows, we often consider the discrete Fourier transform
associated to a function $G\in {\mathcal C}(\T)$, and to avoid any confusion 
we introduce a new notation: let $\mc F_nG:\widehat{\T}_n\to\C$ 
be the discrete Fourier transform of the finite sequence $\{G(\frac x n)\}_{x\in\T_n}$ 
defined similarly to \eqref{eq:fourier} as
\[
\mc F_nG(k):=\sum_{x\in\T_n} G\Big(\frac{x}{n}\Big) e^{-2i\pi xk}, \qquad k \in \widehat{\T}_n.
\]
In particular, we have the Parseval identity
\begin{equation}\label{eq:pars-d}
\sum_{x\in\T_n} G\Big(\frac{x}{n}\Big) \; f_x^\star = \frac{1}{n}\sum_{k\in\widehat{\T}_n}(\mc F_n G)(k) \;  \widehat f^\star(k).
\end{equation}
Furthermore, note that 
$$
\frac 1n \mc F_nG\left(\frac \eta n \right)\ \mathop{\longrightarrow}_{n\to\infty} \ \mc FG(\eta), \qquad \text{for any } \eta \in \Z.
$$

\subsection{Assumptions on initial data} 
Without losing too much of generality, one
can  put natural assumptions on the initial probability measure
$\mu_n(\dd   \mb r,\dd   \mb p)$. 

The first assumption concerns the mean of the initial configurations, 
and  is sufficient in order to derive the first of the hydrodynamic equations \eqref{eq:linear} :

\begin{assumption} \label{ass} 
  \begin{itemize}
  \item The initial total energy can be random but with uniformly bounded expectation:
   \begin{equation}
    \label{012608}
\sup_{n\geqslant 1} \bb E_n\bigg[\frac{1}{n} \sum_{x\in\T_n} {\cal E}_x(0)  \bigg] < +\infty. 
\end{equation}
\item
We assume that there exist continuous initial profiles $r_0:\T\to \R$ and $e_{0}:\T \to (0,+\infty)$ such that 
 \begin{equation} 
\bb E_n[p_x(0)] = 0,\qquad \bb E_n[r_x(0)] = r_0\Big(\frac{x}{n}\Big)
\qquad \text{for any } x \in
\T_n \label{eq:ass-profil} \end{equation}
and for any $G\in \mc C^\infty(\T)$
 \begin{equation} \label{eq:conv-ener-init}
\frac 1n\sum_{x\in\T_n} G\Big(\frac x n\Big) \E_n\big[\mc E_x(0)\big] \xrightarrow[n\to+\infty]{}  \int_\T G(u)\;
    e_0(u)\; \dd u.
\end{equation}
Identity \eqref{eq:ass-profil}, in particular,  implies the mean convergence of the initial elongation:
 \begin{equation} \label{eq:conv-elong-init}
  \frac 1n\sum_{x\in\T_n} G\Big(\frac x n\Big)\E_n \big[ r_x(0) \big]  \xrightarrow[n\to+\infty]{} \int_\T G(u)\;
    r_0(u)\; \dd   u,  
 \end{equation}
 for any $G\in \mc C^\infty(\T)$.
 \end{itemize}
\end{assumption}

\begin{remark}
  By energy conservation \eqref{012608} implies that
  \begin{equation}
    \label{eq:4}
    \sup_{n\geqslant 1} \bb E_n\bigg[\frac{1}{n} \sum_{x\in\T_n} {\cal E}_x(t)  \bigg] < +\infty, \qquad \text{for all } t\ge 0. 
  \end{equation}
\end{remark}
\begin{remark}
  Conditions in \eqref{eq:ass-profil} are assumed in order to simplify
  the proof, but they can be easily relaxed.
\end{remark}

Next assumption is important to obtain the macroscopic equation for the energy in  \eqref{eq:linear}.
It concerns the energy spectrum of fluctuations around the means at initial time.
Define  the {\em initial thermal energy spectrum} 
$\mathfrak{u}_n(0,k),\; k\in \widehat{\T}_n$, as follows: 
let $\widehat{r}(0,k)$ and $\widehat{p}(0,k)$ denote respectively the Fourier transforms 
of the initial random configurations $\{r_x(0)\}_{x\in\T_n}$ and $\{p_x(0)\}_{x\in\T_n}$, 
and let 
\begin{equation}
\label{x1}
\mathfrak{u}_n(0,k):=\frac{1}{2n}\E_n\Big[ \big|\widehat{p}(0,k) \big|^2
+ \big|\widehat{r}(0,k)-\bb E_n[\widehat{r}(0,k)]\big|^2 \Big],\quad k\in \widehat{\T}_n.
\end{equation}
Due to the Parseval identity \eqref{eq:parseval} we have
\[ \frac{1}{n} \sum_{k\in\widehat{\T}_n} \mf u_n(0,k) =\frac{1}{2n}\sum_{x\in\T_n} \E_n\Big[ p_x^2(0) + \big(r_x(0)-\E_n[r_x(0)]\big)^2\Big].\]

\begin{assumption}\label{ass2}
\emph{\bf (Square integrable initial thermal energy spectrum) } 
\begin{equation}
\sup_{n\geqslant 1} \bigg\{\frac{1}{n} \sum_{k\in\widehat{\T}_n} \mathfrak{u}_n^2(0,k)  \bigg\} < +\infty.\label{eq:uniform-int}
\end{equation}
\end{assumption}

This technical assumption can be seen as a way to ensure that the thermal energy does not concentrate on one (or very few) wavelength(s).

\begin{remark} Assumptions \ref{ass} and \ref{ass2} are satisfied if the measures $\mu_n$  are 
given by local Gibbs measures (non homogeneous product), 
corresponding to the given initial profiles of tension and temperature
$\{\tau_0(u),\beta_0^{-1}(u)\; ;\; u\in\T\}$, defined as follows:
\begin{equation}
  \label{eq:gibbs-local}
  \dd \mu^n_{\tau_0(\cdot),\beta_0(\cdot)}\ = \prod_{x\in\T_n} \exp\Big\{-\beta_0\Big(\frac x n \Big) \Big(\mc E_x - \tau_0\Big(\frac x n \Big)r_x\Big) - \mc G_{\tau_0(\frac x  n),\beta_0(\frac x n)}\Big\}\; \dd r_x \dd p_x.
\end{equation}
with $r_0(u) = \tau_0(u)$ and $e_0(u) = \beta_0^{-1}(u) + \frac{r_0^2(u)}2$, see  \cite[Sections 9.2.3--9.2.5]{ko1}.
Note that our assumptions are much more general, as we do not assume any specific 
condition on the correlation structure of $\mu_n$. In particular microcanonical versions of 
\eqref{eq:gibbs-local}, where total energy and total volumes are conditioned at fixed values $ne$ and $nr$, are included by our 
assumptions.
\end{remark}

\begin{remark}
  We will see  that macroscopically,  our assumptions state that the initial energy has a mechanical part, related to $\tau_0(\cdot)$, that concentrates on the 
longest wavelength (i.e.~around $k=0$), see in Section \ref{sec:wigner} equation \eqref{050704} for the precise meaning. For what concerns   the thermal energy, \eqref{eq:uniform-int} states that it has a square integrable density w.r.t.~$k$.
\end{remark}

\subsection{Formulation of mean convergence}

\label{sec3.3}

 In this section we  state two theorems dealing with the mean convergence of the two conserved quantities, 
namely the elongation and energy. 
The first one (Theorem \ref{theo:hydro}) is proved straightforwardly in Section \ref{ssec:proof} below. 
The second one is more involved, and is the main subject of  the present paper.

\begin{theorem}[Mean convergence of the elongation profile]
\label{theo:hydro} Assume that $\{\mu_n\}_{n\in\N}$ is a sequence of probability measures on $\Omega_n$ such that \eqref{eq:ass-profil} is satisfied, with $r_0 \in \mc C(\T)$. Let $r(t,u)$ be the solution defined on $\R_+\times\T$ of the linear diffusive equation:
\begin{equation}\left\{\begin{aligned}
\partial_t r(t,u)   & = \frac{1}{2\gamma} \partial_{uu}^2 r(t,u), \qquad   (t,u)\in \R_+\times \T, \\ r(0,u)&=r_0(u).
\end{aligned}
\right.\label{eq:elong-evol}\end{equation}
Then, for any $G \in \mc C^\infty(\T)$ and $t\in \bb R_+$,
\begin{align}\label{eq:conv-p}
&\lim_{n\to +\infty} \frac{1}{n}\sum_{x\in\T_n}G\Big(\frac{x}{n}\Big) \E_n\big[p_x(t)\big]  = 0,\\ \label{eq:conv-r}
&\lim_{n\to+ \infty} \frac{1}{n}\sum_{x\in\T_n}G\Big(\frac{x}{n}\Big) \E_n\big[r_x(t)\big]  = \int_{\T} G(u) \;r(t,u)\; \dd u.
\end{align}
\end{theorem}



\begin{theorem}[Mean convergence of the empirical profile of energy]

\label{theo:hydro1} Let $\{\mu_n\}_{n\in\N}$ be a sequence of probability measures on $\Omega_n$ such that 
Assumptions \ref{ass} and \ref{ass2} are satisfied.
Then, for any smooth function $G:\R_+\times \T \to \R$ compactly supported with respect to the time variable $t\in\R_+$, we have
\begin{equation}\label{eq:convergence}
\lim_{n\to +\infty} \frac{1}{n}\sum_{x\in\T_n}\int_{\R_+} G\Big(t,\frac{x}{n}\Big) \E_n\big[\mc E_x(t)\big] \; \dd t = \int_{\R_+ \times \T} G(t,u) \;e(t,u)\; \dd t\dd u,
\end{equation}
where $e(t,u)=e_{\mathrm{mech}}(t,u) +
e_{\mathrm{thm}}(t,u)$, with 
\begin{itemize}
\item the \emph{mechanical energy}, given by
  $e_{\textrm{\emph{mech}}}(t,u) :=\frac{1}{2} \left(r(t,u)\right)^2
  $ and  the function $r(t,u)$ being the
solution of \eqref{eq:elong-evol},
\item the \emph{thermal energy}
$e_{\mathrm{thm}}(t,u) $, defined as the solution of
\begin{equation}\label{eq:energy-evol}\left\{\begin{aligned}
\partial_t e_{\mathrm{thm}}(t,u) & = \frac{1}{4\gamma}\partial_{uu}^2 e_{\mathrm{thm}}(t,u) + \frac{1}{2\gamma}\big(\partial_u r(t,u)\big)^2,\\
e_{\mathrm{thm}}(0,u) & = \beta^{-1}_0(u) =e_0(u) - e_{\rm mech}(0,u)>0.
\end{aligned}\right.\end{equation}
\end{itemize}
\end{theorem}

The proof of
Theorem \ref{theo:hydro1} is the aim of Sections \ref{sec:wigner} -- \ref{sec:proof}.

\begin{remark}
Note that \eqref{eq:elong-evol} and \eqref{eq:energy-evol} are equivalent to the system \eqref{eq:linear}. This new way of seeing the macroscopic equations is more convenient, as it naturally arises from the proof.
More precisely, using \eqref{eq:elong-evol} we conclude that the mechanical
energy $e_{\rm mech}(t,u)$ satisfies the equation
\begin{equation*}
\partial_t e_{\rm mech}(t,u)=\frac{1}{2\ga}\Big(\partial_{uu}^2e_{\rm
    mech}(t,u)-\big(\partial_u r(t,u)\big)^2\Big)
\end{equation*}
and  the macroscopic energy density function satisfies
\begin{equation*}\left\{\begin{aligned}
\partial_t e(t,u) & = \frac{1}{4\gamma}\partial_{uu}^2\big( e(t,u) + e_{\rm mech}(t,u)\big),\\
e(0,u) & = e_0(u).
\end{aligned}\right.\end{equation*}
\end{remark}

\begin{remark}
  We actually prove a stronger result that includes a local equilibrium statement, see
Theorem \ref{theo:limitW} below.
\end{remark}

\subsection{Proof of the hydrodynamic limit  for the elongation} \label{ssec:proof}
Here we give a simple proof of Theorem \ref{theo:hydro}. From the evolution equations \eqref{eq:dynamics} we have the
  following identities:
$$
      \frac 1n \sum_{x\in\T_n}G\Big(\frac{x}{n}\Big) \E_n\big[r_x(t) -
      r_x(0)\big] = n^2 \int_0^t\frac 1n
      \sum_{x\in\T_n}G\Big(\frac{x}{n}\Big) \E_n\big[p_x(s) -
      p_{x-1}(s)\big]  \dd s
$$
   and
 \begin{align*}
    2\ga n^2 \int_0^t\frac 1n
      \sum_{x\in\T_n}G\Big(\frac{x}{n}\Big) \E_n\big[p_x(s)\big]  \dd s  = \; & n^2 \int_0^t\frac 1n
      \sum_{x\in\T_n}G\Big(\frac{x}{n}\Big) \E_n\big[r_{x+1}(s) -
      r_{x}(s)\big]  \dd s\\
&
+\frac 1n \sum_{x\in\T_n}G\Big(\frac{x}{n}\Big) \E_n\big[p_x(0) -
      p_x(t)\big].
 \end{align*}
Substituting from the second equation into the first one we conclude that
  \begin{align}
  \label{020203}
      \frac 1n \sum_{x\in\T_n}G\Big(\frac{x}{n}\Big) \E_n\big[r_x(t) - r_x(0)\big] = 
      &\int_0^t \frac 1{2\gamma n} \sum_{x\in\T_n} \Delta_n G\Big(\frac{x}{n}\Big) \E_n\big[r_x(s)\big] \dd s\\
      &- \frac 1{2\gamma n^2} \sum_{x\in\T_n} \nabla_n G\Big(\frac{x}{n}\Big) \E_n\big[p_x(0) - p_x(t)\big],\nonumber
    \end{align}
where 
\begin{align*}\nabla_n G\Big(\frac{x}{n}\Big)& = n\Big(G\Big(\frac{x+1}{n}\Big) - G\Big(\frac{x}{n}\Big)\Big)\\
 \Delta_n G\Big(\frac{x}{n}\Big) &= n\Big(\nabla_n G\Big(\frac{x}{n}\Big)  - \nabla_n
   G\Big(\frac{x-1}{n}\Big)\Big).
\end{align*}
By energy conservation and Assumption \ref{ass} it is easy to see that 
\begin{align}
\label{030203}
    \bigg|\frac 1{n^2} \sum_{x\in\T_n} \nabla_n G\Big(\frac{x}{n}\Big) \E_n\big[p_x(t)\big]\bigg|^2 & \le
    \frac 1{n^2} \bigg(\frac 1n \sum_{x\in\T_n} \Big|\nabla_n G\Big(\frac{x}{n}\Big)\Big|^2\bigg) \bigg(\frac 1n \sum_{x\in\T_n} \E_n\big[p_x^2(t)\big]\bigg)\nonumber\\
    &\le \frac {C(G)}n \bigg(  \frac 1n \sum_{x\in\T_n} \E_n\big[\mathcal E_x^2(0)\big]  \bigg)\mathop{\longrightarrow}_{n\to\infty} \ 0.
  \end{align}
Let us define
$$
\bar r^{(n)}(t,u):=\E_n\big[r_x(t)\big], \quad \text{ for any } u\in
\left[\tfrac{x}{n},\tfrac{x+1}{n}\right), \quad n\ge1.
$$
Thanks to the energy conservation we know that
there exists $R>0$ such that
\begin{equation}
\label{010203}
\sup_{n\ge1}\sup_{t\in[0,T]}\|\bar r^{(n)}(t,\cdot)\|_{{\bf L}^2(\T)}=:R<+\infty.
\end{equation}
The above means that   for each $t\in[0,T]$ the sequence $\left\{\bar r^{(n)}(t,\cdot)\right\}_{n\ge1}$ is contained in $\bar B_R$ -- the closed ball of radius $R>0$ in ${\bf L}^2(\T)$, centered at $0$.
 The ball  is compact in ${\bf L}^2_w(\T)$ --
the space of square integrable functions on the torus $\T$ equipped
with the weak ${\bf L}^2$ topology.  The  topology restricted to $\bar B_R$ is metrizable, with the respective metric given e.g.~by
$$
d(f,g):=\sum_{n=1}^{+\infty}\frac{1}{2^n}\frac{|\langle f-g,\phi_n\rangle_{{\bf L}^2(\T)}|}{1+|\langle f-g,\phi_n\rangle_{{\bf L}^2(\T)}|},\quad f,g\in \bar B_R,
$$
where $\{\phi_n\}$ is a countable and dense subset of  ${\bf L}^2(\T)$ that can be chosen of elements of ${\cal C}^\infty(\T)$.
From \eqref{020203} and \eqref{030203} we conclude in particular that  for each $T>0$ the sequence  $\left\{\bar r^{(n)}(\cdot)\right\}$ is equicontinuous 
 in ${\cal C}\left([0,T],\bar B_R\right)$. Thus, according to the Arzela Theorem, see e.g.~\cite[p. 234]{kelley}, it is
sequentially pre-compact in the space ${\cal C}\left([0,T],{\bf L}^2_w(\T)\right)$ for any $T>0$. Consequently, any limiting point of the
sequence satisfies  the partial differential equation
\eqref{eq:elong-evol}  in a weak sense 
in the class of ${\bf L}^2(\T)$ functions.  
Uniqueness of the weak solution  of the heat equation gives the
convergence claimed in \eqref{eq:conv-r} and the identification of the
limit as the strong solution of \eqref{eq:elong-evol}. 

Concerning \eqref{eq:conv-p}, from \eqref{eq:dynamics} we have
\begin{align*}
    \frac{1}{n}\sum_{x\in\T_n}G\Big(\frac{x}{n}\Big) \E_n\big[p_x(t)\big] = &
    \frac{e^{-2\gamma n^2 t} }{n}\sum_{x\in\T_n}G\Big(\frac{x}{n}\Big) \E_n\big[p_x(0)\big] \\
   & + \int_0^t e^{-2\gamma n^2 (t-s)} \sum_{x\in\T_n}\nabla^*_n G\Big(\frac{x}{n}\Big) \E_n\big[r_x(s)\big] \dd s, 
  \end{align*}
  where $\nabla_n^* G(\frac{x}{n}) = n\left(G(\frac{x}{n}) - G(\frac{x-1}{n})\right)$. 
Using again energy conservation and the Cauchy-Schwarz inequality, it is easy to see that the right hand side of the above vanishes as $n\to\infty$.

\begin{remark} Note that we have not used the fact that the initial average of the velocities vanishes. Additionally,  by standard methods it is possible 
 to obtain the convergence of elongation and  momentum empirical distributions 
in probability (see \eqref{eq:conv-p} and \eqref{eq:conv-r}), but we shall not pursuit this here.\end{remark}

\section{Conjecture for anharmonic interaction
 and thermodynamic considerations}
\label{sec:conj-anharm-inter}

Our results concern only harmonic interactions, but we can state the expected macroscopic behavior 
for the anharmonic case. Consider a non-quadratic potential $V(r)$, of class $\mc C^1$ and growing fast 
enough to $+\infty$ as $|r|\to\infty$. The dynamics is now defined by
\begin{equation} \left\{ \begin{aligned}
    \dd   r_x(t) &= n^2 \big(p_x(t) - p_{x-1}(t)\big)\; \dd   t\\
    \dd   p_x(t) &= n^2 \big(V'(r_{x+1}(t)) - V'(r_x(t))\big)\; \dd   t - 2p_x(t^-)
    \; \dd   \mathcal N_x(\gamma n^2t), \; x\in\T_n.
  \end{aligned}\right.
\label{eq:nl-dynamics}
\end{equation} 
The stationary measures are given by the canonical Gibbs distributions
\begin{equation}
  \label{eq:gibbs}
  d\mu^n_{\tau,\beta} = \prod_{x\in\T_n} e^{- \beta (\mc E_x - \tau
      r_x) - \mc G_{\tau,\beta}}\; \dd r_x\; \dd p_x, \qquad \tau \in \mathbb{R}, \beta>0, 
\end{equation}
where we denote 
$$
\mc E_x = \frac{p_x^2}{2} + V(r_x),
$$
the energy that we attribute to the particle $x$, and 
\begin{equation}
  \label{eq:pfunct}
  \mc G_{\tau,\beta}= \log \left[\sqrt{2\pi \beta^{-1}}\int e^{-\beta(V(r) - \tau r)}\; \dd r  \right].
\end{equation}
Thermodynamic entropy $S( r, e)$ is defined as
\begin{equation}
  \label{eq:S}
  S(r, e) = \inf_{\tau\in \R,\beta>0} \big\{ \beta e - \beta\tau  r  +  \mc G(\tau,\beta)\big\}.
\end{equation}
Then we obtain the inverse temperature and tension as functions of the volume $r$ and internal energy $u$:
\begin{equation}
\bs\beta^{-1}(r,e) = \partial_{e} S(r,e) , \qquad \bs\tau(r,e) = - \bs\beta^{-1}(r,e)\partial_{r} S(r,e) 
\label{eq:5}
\end{equation}
The macroscopic profiles of elongation $r(t,u)$ and energy $e(t,u)$ will satisfy the equations
\begin{equation}
  \label{eq:6-b}
  \begin{split}
     \partial_t r  &=  \frac{1}{2\gamma} \partial^2_{uu} \big[\bs\tau(r, e)\big],\\
      \partial_t e &= \partial_u \big[\mathcal D(r,e) \partial_u  \bs\beta^{-1} \big] 
      + \frac 1{4\gamma}\partial^2_{uu}\left(\bs\tau(r,e)^2\right).
  \end{split}
\end{equation}
Here the diffusivity $\mathcal D(r,e)>0$ is defined by a Green-Kubo formula for the infinite dynamics in equilibrium at the given values $(r,e)$. The precise definition and the proof of the convergence of Green-Kubo formula 
for this dynamics can be found in \cite{bo2}. 

A straightforward calculation gives the expected increase of thermodynamic entropy:
\begin{equation}
  \label{eq:2ndP}
  \frac d{dt} \int_\T S(r(t,u), e(t,u)) \; \dd u = \int_\T \bs \beta\left(\frac{(\partial_u \bs\tau)^2}{2\gamma} 
  + \mathcal D(r,e)\big( \partial_u  \bs\beta^{-1}\big)^2\right) \; \dd u \ge 0.
\end{equation}

\section{Time-dependent Wigner distributions}
\label{sec:wigner}

Before exposing the strategy of the proof of Theorem \ref{theo:hydro1}, let us start by introducing our main tool: the Wigner distributions associated to the dynamics.

\subsection{Wave function for the system of oscillators}
\label{sec:wave}

Let $\widehat{p}(t,k)$ and $\widehat{r}(t,k)$, for $k\in \widehat{\T}_n$,  denote
the Fourier transforms of, respectively, the momentum and elongation components of the microscopic configurations
$\{p_x(t)\}_{x\in\T_n}$ and $\{r_x(t)\}_{x\in\T_n}$, as in \eqref{eq:fourier}.
Since they are real valued we have, for any $k \in \widehat{\T}_n$,
\begin{equation}
\label{x2}
\widehat p^\star(t,k)=\sum_{x\in\T_n}e^{2\pi i kx}p_x(t)=\widehat p(t,-k)\
\mbox{ and likewise }\ \widehat r^\star(t,k)=\widehat r(t,-k).
\end{equation}
 The {\em wave function} associated to the dynamics is defined as 
 \begin{equation*}
{\psi}_{x}(t):=   r_x(t) + i p_x(t),\qquad x\in \T_n.
\end{equation*}
Its Fourier transform equals
\begin{equation*}
\widehat\psi(t,k) :=  \widehat{r}(t,k)+
i\widehat{p}(t,k),\qquad k\in \widehat{\T}_n.\end{equation*} 
Taking into account \eqref{x2} we obtain
\begin{align*}
\widehat p(t,k)&=\frac{1}{2i}\big(\widehat\psi(t,k)- \widehat\psi^\star(t,-k)\big),\\
\widehat r(t,k)&=\frac{1}{2}\big(\widehat\psi(t,k)+ \widehat\psi^\star(t,-k)\big).
\end{align*}
With these definitions we have $|{\psi}_{x}|^2 = 2
\mathcal{E}_x$ and the initial thermal energy
spectrum, defined in \eqref{x1}, satisfies
\begin{equation}
  \label{x111}
 \mf u_n(0,k)=\widetilde{\mf u}_n(0,k) +{\rm Im}\big[{\rm Cov}_n(\widehat p(0,k),\widehat r(0,k))\big],\qquad k\in \widehat{\T}_n.
\end{equation} 
Here ${\rm Cov}_n(X,Y):=\E_n[XY^\star]-\E_n[X]\,\E_n[Y^\star]$ is the covariance of complex random variables  $X$ and $Y$,
and
$$
\widetilde{\mf u}_n(0,k):=\frac1{2n}\E_n\Big[\big|\widehat\psi(0,k)-\E_n[\widehat\psi(0,k)]\big|^2\Big]
$$
Using Cauchy-Schwarz inequality we conclude easily that
$$
 \frac12\widetilde{\mf u}_n(0,k) \le \mf u_n(0,k)\le 2\widetilde{\mf u}_n(0,k),\quad k\in \widehat{\T}_n.
$$
Therefore, \eqref{eq:uniform-int} is equivalent with
\begin{equation}
\label{eq:uniform-int1}
\sup_{n\geqslant 1} \bigg\{\frac{1}{n} \sum_{k\in\widehat{\T}_n} \widetilde{\mathfrak{u}}_n^2(0,k)  \bigg\} < +\infty.
\end{equation} 
After a straightforward calculation, the equation that governs the time evolution of the wave function can be deduced 
from \eqref{eq:dynamics} as follows:
\begin{align}
\dd   \widehat\psi(t,k) = 
&
-n^2 \Big(2i \sin^2(\pi k) \widehat\psi(t,k)+\sin(2\pi k)\widehat\psi^\star(t,-k)\Big) \; \dd   t
\notag\\
&-\frac{1}{n}\sum_{k'\in\widehat{\T}_n} \Big\{ \widehat\psi(t^-,k-k')-\widehat\psi^\star(t^-,k'-k)\Big\} \;  \dd
\widehat{\mathcal{N}}(  t,  k'),\label{eq:dynamics-psi}
\end{align}
with   initial condition  $\widehat\psi(0,k)=\widehat r(0,k).$
The semi-martingales $\big\{\widehat{\mathcal{N}}(  t,  k) \; ; \;t\geq 0 \big\}$ are defined as  
\[
\widehat{\mathcal{N}}(  t,  k):= \sum_{x\in\T_n}  {\mc N}_x(\gamma n^2 t)
e^{-2i\pi x k}, \qquad k \in \widehat{\T}_n. \]
Observe that we have 
$\widehat{\mathcal{N}}^\star(  t,  k)=\widehat{\mathcal{N}}(  t, -k)$. In addition,
its mean and covariance  equal respectively
\begin{eqnarray*}
&&
\big\langle \dd\widehat{\mathcal{N}}(  t,  k)\big\rangle=\gamma n^3 \; \delta_{k,0}\; \dd t,\\
&&
\big\langle \dd\widehat{\mathcal{N}}(  t,  k),\dd\widehat{\mathcal{N}}(t,k')\big\rangle=\gamma n^3 t\; \delta_{k,-k'}\; \dd t,
\end{eqnarray*}
where $\delta_{x,y}$ is the usual Kronecker delta function, which equals 1 if $x=y$ and 0 otherwise.
The conservation of energy, and Parseval's identity, imply together that:
\begin{equation}\label{eq:conservation}
\big\|\widehat\psi(t)\big\|_{\mb L^2}
=\big\|\widehat\psi(0)\big\|_{\mb L^2}\quad  \mbox{for all  } t \ge 0.
\end{equation}

%

\subsection{Wigner distributions and Fourier transforms} 
 The \emph{Wigner distribution} $\mb W^+_n(t)$ corresponding to the wave function $\psi(t)$ 
is a  distribution defined by its action on smooth functions $G \in \mc C^\infty(\T\times \T)$ as 
\begin{equation}
\label{eq:wigner-distrib}
\big\langle \mb W^+_n(t),\; G \big\rangle := \frac1n
\sum_{k\in\widehat{\T}_n}\sum_{\eta\in\Z} W^+_n(t,\eta,k) (\mc
FG)^\star(\eta,k),
\end{equation}
where the \emph{Wigner function} $W_n^+(t)$ is given  for any $(k,\eta)\in\widehat{\T}_n\times \Z$ by
\begin{equation}
\label{010803}
 W_n^+(t,\eta,k) := \frac{1}{2n} \bb E_n\Big[
 \widehat\psi\Big(t,k+\frac{\eta}{n}\Big)
 \widehat\psi^\star(t,k)\Big].
\end{equation}
Here, we use the mapping $\bb Z \ni\eta \mapsto \frac{\eta}{n} \in \widehat \T_n$, and $\mathcal F G(\eta, v)$ 
denotes the Fourier transform with respect to the first variable.  

\begin{remark}Note that this definition of the Wigner function is not the standard symmetric one. Indeed, since the setting here is discrete, it turns out that \eqref{010803} is the convenient way to identify the Fourier modes, otherwise we would have worked with ill-defined quantities, for instance $\frac{\eta}{2n}$, which are not always integers.
\end{remark}

The main interest of the Wigner distribution is that mean convergence of the empirical energy profile \eqref{eq:convergence} 
can be restated in terms of convergence of Wigner functions 
(more precisely, their Laplace transforms, see Theorem \ref{theo:limitW} below), 
thanks to the following identity: if $G(u,v)\equiv G(u)$ does not depend on the second  variable $v\in\T$, then 
\begin{equation}
\label{022303}
\big\langle \mb W^+_n(t),G \big\rangle =\frac1n \sum_{x\in\T_n} \bb E_n\big[\mc E_x(t)\big]G\Big(\frac{x}{n}\Big) . 
\end{equation}
 Indeed,
 from \eqref{eq:wigner-distrib} we obtain then
\begin{equation}
\label{022303a}
\big\langle \mb W^+_n(t),G \big\rangle = \frac{1}{2n^2}
\sum_{k\in\widehat{\T}_n}\sum_{x,x'\in \T_n}\sum_{\eta\in\Z}\E_n\big[\psi_{x'}^\star (t)\psi_x(t)\big]e^{2\pi i (x'-x)k}e^{2\pi ix \frac \eta n}(\mathcal F G)^\star(\eta).
\end{equation}
Performing the summation over $k$ we conclude that the right hand side equals
\begin{equation}
\label{022303b}
 \frac{1}{2n}
\sum_{x,x'\in \T_n}\sum_{\eta\in\Z}\E_n\big[\psi_{x'}^\star (t)\psi_x(t)\big]1_{\Z}\Big(\frac{x'-x}{n}\Big)e^{2\pi i x\frac\eta n}(\mathcal F G)^\star(\eta),
\end{equation}
where $1_{\Z}$ is the indicator function of the set of integers. Since $0\le x,x'\le n-1$ and $|\psi_x(t)|^2=2{\cal E}_x(t)$ we conclude that the above expression equals
\begin{equation}
\label{022303c}
 \frac{1}{n}
\sum_{x\in \T_n}\sum_{\eta\in\Z}\E_n\big[{\cal E}_x(t)\big]e^{2\pi ix\frac \eta n}(\mathcal F G)^\star(\eta)=\frac{1}{n}\sum_{x\in \T_n}\E_n\big[{\cal E}_x(t)\big]G\left(\frac{x}{n}\right)
\end{equation}
and \eqref{022303} follows.
%
%
%
Identity \eqref{022303} can be also interpreted as follows:  the $k$-average of the Wigner distribution gives the Fourier transform of the energy:
\begin{equation}
  \label{eq:2}
  \frac 1n \sum_{k\in\widehat{\T}_n} W_n^+(t,\eta,k) = \frac 1{2n} \sum_{x\in\T_n} e^{-2\pi i x\frac \eta n}\;  {\bb E_n \big[}|\psi_x(t)|^2{\big]} 
  = \frac 1n \; {\bb E_n\bigg[}\widehat{\mc E}\left(t,\frac{\eta}{n}\right){\bigg]},\quad \eta\in\Z.
\end{equation}
%
To close the equations governing the evolution of $\mb W_n^+(t)$, we need to define three other Wigner-type functions. We let
\begin{align}
 W_n^-(t,\eta,k)& :=  \frac{1}{2n} \bb E_n\Big[ \widehat\psi^\star\Big(t,-k-\frac{\eta}{n}\Big) \widehat\psi(t,-k)\Big]
=(W_n^+)^\star(t,-\eta,-k),\label{eq:w-}\\
 Y_n^+(t,\eta,k) & := \frac{1}{2n} \bb E_n\bigg[\widehat\psi\Big(t,k+\frac{\eta}{n}\Big) \widehat\psi(t,-k)\bigg], \label{eq:y+}\\
 Y_n^-(t,\eta,k) & :=\frac{1}{2n} \bb E_n\bigg[ \widehat\psi^\star\Big(t,-k-\frac{\eta}{n}\Big) \widehat\psi^\star(t,k)\bigg]
= \big(Y_n^+\big)^\star(t,-\eta,-k).\label{eq:y-}
\end{align}
\begin{remark}Note that an easy change of variable gives
\begin{equation}
\label{eq:sumequal}
\frac{1}{n}\sum_{k\in \widehat\T_n} W_n^+(t,\eta,k)=\frac{1}{n}\sum_{k\in\widehat\T_n}W_n^-(t,\eta,k), 
\end{equation}
for any $t\geqslant 0$ and $\eta \in \bb Z$. This remark will be useful in what follows.\end{remark}

Thanks to identity \eqref{022303}, we have reduced the proof of Theorem \ref{theo:hydro1} -- and more precisely the proof of convergence \eqref{eq:convergence} --  to the investigation of the Wigner sequence $\{\mb W_{n}^+(\cdot), \mb Y_{n}^+(\cdot), \mb Y_{n}^-(\cdot), \mb W_{n}^-(\cdot)\}_{n}$. The next sections are split as follows:
\begin{enumerate}
\item we first prove that this last sequence is pre-compact (and therefore admits a limit point) in Section \ref{ssec:weak};
\item then, we characterize that limit point in several steps: \begin{enumerate}
\item we write a decomposition of the Wigner distribution into its mechanical and thermal parts in Section \ref{ssec:decomp1};
\item the convergence of the mechanical part is achieved in Section \ref{ssec:asympw} and Section \ref{sec:wign-distr-assoc};
\item to solve the thermal part, we need to take its Laplace transform in Section \ref{sec:lap1}, and then to study the dynamics it follows in Section \ref{sec:time} and Section \ref{sec:lap};
\item the convergence statements for this Laplace transform are given in Section \ref{sec:proof}, the main results being Proposition \ref{lem:mech} and Proposition \ref{lemma3};
\item finally, we go back to the convergence of the Wigner distributions in Theorem \ref{theo:limitW}, and then to the conclusion of the proof of Theorem \ref{theo:hydro1} in Section  \ref{sec:end}.
\end{enumerate}

\end{enumerate}

\subsection{Properties of the Wigner distributions}

\subsubsection{Weak convergence}\label{ssec:weak}

From \eqref{eq:4} we have that, for any $n \geqslant 1$ and $G \in \mc C^\infty(\T \times \T)$, 
\begin{align}
\label{061407}
\big|\big\langle &\mb W^+_n(t), G\big\rangle\big| \\
& \leqslant 
\frac{1}{(2n)^2}\sup_{\eta\in\Z} \bigg\{\sum_{k\in\widehat{\T}_n}
\bb E_n\Big[ \Big|\widehat\psi\Big(t,k+\frac{\eta}{n}\Big)\Big|^2\Big]\bigg\}^{\frac12} 
\bigg\{ \sum_{k\in\widehat{\T}_n}\bb E_n\Big[ \big|\widehat\psi(t,k)\big|^2\Big]\bigg\}^{\frac12} \|G\|_0\nonumber\\
& \le \frac{\|G\|_0}{2n}\; \bb
  E_n\bigg[\sum_{x\in\T_n} {\cal E}_x(t)\bigg] \le C \|G\|_0 , \nonumber
\end{align}
where the norm $ \|G\|_0 $ is defined by \eqref{022104}.
Hence, for the corresponding dual norm, we have the bound 
\begin{equation}
\label{eq:bound}
\sup_{t \geqslant 0}\sup_{n\geqslant 1}\bigg\{\big\|\mb W^+_n(t)\big\|_{0}' + \big\|\mb W^-_n(t)\big\|_{0}'+\big\|\mb Y^+_n(t)\big\|_0'+\big\|\mb Y^-_n(t)\big\|_0'\bigg\} \leqslant 4C,
\end{equation}
which implies weak convergence. Note that condition \eqref{eq:conv-ener-init} ensures that, if $G(u,v)\equiv
G(u)$ at the initial time $t=0$, then we have   
\begin{align*}
\lim_{n\to+\infty}  \big\langle \mb W^\pm_n(0), G\big\rangle & = \int_{\T} e_0(u) \;  G^\star(u) \; \dd   u, 
\end{align*}

\subsubsection{Decomposition into the mechanical and fluctuating part} \label{ssec:decomp1}

We decompose the  wave function into its
mean w.r.t.~${\bb E}_n$ and its fluctuating part, as follows:
\begin{equation*}
\psi_{x}(t)=\overline\psi_{x}(t) +\widetilde\psi_{x}(t), \qquad x \in \T_n,\; t \ge 0. 
\end{equation*}
Notice that for the initial data we have $\overline\psi_{x}(0)= r_0\left(\frac{x}{n}\right)$. 
It need not be true for $t>0$.
The  Fourier transform of the sequences $\{\overline\psi_{x}(t)\}$ and $\{\widetilde\psi_{x}(t)\}$ 
shall be denoted by $\widehat{\overline\psi}(t,k)$ and $\widehat{\widetilde\psi}(t,k)$.
The deterministic function $\widehat{\overline\psi}(t,k)$ satisfies the autonomous equation:
\begin{align}
\dd   \widehat{\overline\psi}(t,k) = 
&
-n^2 \Big(2i \sin^2(\pi k) \widehat{\overline\psi}(t,k)+\sin(2\pi k)\widehat{\overline\psi}^\star(t,-k)\Big) \; \dd   t
\notag\\
&- n^2\gamma\; \Big\{ \widehat{\overline\psi}(t,k)-\widehat{\overline\psi}^\star(t,-k)\Big\} \;  \dd t,\label{eq:dynamics-overlinepsi}
\end{align}
with   initial condition  $\widehat{\overline\psi}(0,k)=(\mc F_n r_0)(k).$

\medskip

The Wigner distribution $\mb W_n^+(t)$ can be decomposed accordingly as follows:
\begin{equation}
\mb W_n^+(t) = \overline{\mb W}^+_n(t)+ \widetilde{\mb W}_n^+(t), \label{eq:decompo}
\end{equation}
where the Fourier transforms of $\overline{\mb W}_n^+(t)$ and $\widetilde{\mb W}_n^+(t)$ are given by
\begin{align}
\overline{W}_n^+(t,\eta,k)&:= \frac{1}{2n} \widehat{\overline \psi}\Big(t,k+\frac{\eta}{n}\Big)\; \widehat{\overline \psi}^\star(t,k) 
\label{eq:Wbar}\\
\widetilde{W}_n^+(t,\eta,k)&:= \frac{1}{2n}
\E_n\bigg[\widehat{\widetilde \psi}\Big(t,k+\frac{\eta}{n}\Big)\; {\widehat{\widetilde \psi}}^\star(t,k)\bigg],\label{eq:Wtilde}
\end{align}
for any $(\eta,k)\in\Z\times \widehat{\T}_n$.

\medskip

At initial time $t=0$, to simplify notations we write $\overline{W}_n^+(\eta,k):=
 \overline{W}_n^+(0,\eta,k)$ and $\widetilde{W}_n^+(\eta,k):=
 \widetilde{W}_n^+(0,\eta,k)$. Note that the mean part is completely explicit: we   have
 \begin{equation}
\overline{W}_n^+(\eta,k)  =\frac{1}{2n}({\cal F}_nr_0)\left(k+\frac{\eta}{n}\right)({\cal F}_nr_0)^\star(k), \label{eq:equality}
\end{equation}
Recalling \eqref{x111}, the initial thermal energy spectrum can be rewritten as
\[
\mathfrak{u}_n(0,k)=\frac{1}{2n}\E_n\Big[ \big|\widehat{\widetilde\psi}(0, k) \big|^2 \Big].
\]
Reproducing the decomposition \eqref{eq:y+}  one can easily write similar definitions for
 $\overline{\mb W}_n^-$, $\widetilde{\mb W}_n^-$ and the respective  $\overline{\mb Y}_n^\pm$, $\widetilde{\mb Y}_n^\pm$
 distributions. Due to the fact that $\widebar \psi_x (0)= r_0(\frac x n)$
 and is  real-valued (and extending the convention of omitting the
 argument $t=0$ for all Wigner-type distributions) we have
 \begin{equation}
 \overline{W}_n^-(\eta,k) = \overline{W}_n^+(\eta,k)=
 \overline{Y}_n^+ (\eta,k)= \overline{Y}_n^-(\eta,k)=W_n(r_0\; ;\; \eta,k), \label{eq:allequal}
\end{equation}
where we define
\begin{equation}
  \label{eq:Wn}
  W_n(r\; ;\; \eta,k) :=\frac{1}{2n}({\cal F}_nr)\left(k+\frac{\eta}{n}\right)({\cal F}_nr)^\star(k),
\end{equation}
for any $r\in {\cal C}(\T)$ and $(\eta,k)\in\Z\times\widehat{\T}_n$.

\subsubsection{Asymptotics of $\widetilde{W}_n^+$} \label{ssec:asympw}

Throughout the remainder of the paper we shall use the following
notation: given a function $f:\widehat{\T}_n\to\mathbb C$ we
denote its $k$-average by
\begin{equation}
\label{square}
\big[f(\cdot)\big]_n:=\frac1n\sum_{k\in \widehat{\T}_n}f(k).
\end{equation} 
The initial fluctuating Wigner function 
is related, as $n\to+\infty$, to the initial thermal energy 
$e_{\rm thm}(0,u)= e_0(u) - r_0^2(u)/2$ as follows:
\begin{equation}
\label{eq:convergence_tilde}
\left[ \widetilde W_n^+(\eta,\cdot)\right]_n \xrightarrow[n\to\infty]{} \big(\mc F e_{\rm thm}(0,\cdot)\big)(\eta), 
\qquad \text{for any } \eta \in \bb Z.
\end{equation}
The last convergence follows from Assumption \ref{ass} and from an explicit computation that yields:
\[
\left[ \widetilde W_n^+(\eta,\cdot)\right]_n= 
\frac{1}{n}\sum_{x\in\T_n} \bb E_n\big[\mc E_x(0)\big] e^{-2\pi i \eta \frac x n}
- \mc F_n\Big(\frac{r_0^2}{2}\Big)\Big(\frac \eta n\Big).
\]
In addition, Assumption \ref{ass2} on the initial spectrum (see \eqref{eq:uniform-int}) implies that
\begin{equation}
\label{071407}
w_*:=\sup_{n\ge1}{\sup_{\eta\in\bb Z}} \Bigg\{ \sum_{\iota\in\{-,+\}}
\left[\left|\widetilde W_n^\iota(\eta,\cdot)\right|^2 +\left|\widetilde Y_n^\iota(\eta,\cdot)\right|^2\right]_n\Bigg\} <+\infty.
\end{equation}

\subsubsection{Wigner distribution associated to a macroscopic smooth profile}
\label{sec:wign-distr-assoc}

Similarly to \eqref{eq:Wn}, given a continuous real-valued function $\{r:=r(u)\; ; \; u\in \T \}$, define 
\begin{equation}
\label{040704}
W(r\; ;\; \eta,\xi) :=\frac12 ({\mc F} r)(\xi+\eta) (\mc F r)^\star(\xi),\qquad (\eta, \xi)\in\Z^2.
\end{equation}
Observe that, for any $\eta \in \Z$,
\begin{equation}
  \label{eq:s1}
  \sum_{\xi\in\Z} W(r\; ;\; \eta,\xi) = \frac 12 {\cal F}(r^2) (\eta).
\end{equation}
\begin{proposition}
\label{propo021004}
Suppose that $r\in {\cal C}(\T)$ and $G\in{\cal C}^\infty(\T\times \T)$. Then
\begin{equation}
\label{050704b}
\begin{split}
  \lim_{n\to+\infty}\sum_{\eta\in\Z}\Big[W_n(r\; ;\; \eta,\cdot)
  ({\cal F}G)^\star(\eta,\cdot)\Big]_n &=\sum_{\eta\in\Z}
  \bigg(\sum_{\xi\in \Z} W(r\; ; \; \eta,\xi) \bigg) ({\cal
    F}G)^\star(\eta,0)\\
 & = {\frac 12} \int_\T r^2(u) {G^\star}(u,0) \; \dd u .
\end{split}
\end{equation}

\end{proposition}

\begin{proof} We prove the proposition under the assumption that
  $r\in {\cal C}^\infty(\T)$. The general case can be obtained by an
  approximation of a continuous initial profile by a sequence of
  smooth ones.

Using the dominated convergence theorem we conclude that the expression 
on the left hand side of \eqref{050704b} equals 
$\sum_{\eta\in\Z} \lim_{n\to+\infty} f_n(\eta)$, with
\begin{equation}\label{eq:feta}
f_n(\eta):=\Big[ \overline{W}_n^+(\eta,\cdot) ({\cal F}G)^\star(\eta,\cdot)\Big]_n,\qquad \eta\in\Z.
\end{equation}
This can be written as
$$
f_{n}(\eta)=\frac{1}{2n^2}\sum_{k\in\widehat{\T}_n}\sum_{x,x'\in\T_n}e^{-2\pi i \eta\frac{x}n}e^{2\pi ik(x'-x)}r\Big(\frac{x}{n}\Big)
r\Big(\frac{x'}{n}\Big) ({\cal F}G)^\star(\eta,k).
$$
Using Fourier representation (see \eqref{eq:fourierZZ}), we can write $f_{n}(\eta)$ as
$$
\frac{1}{2n^2}\sum_{\xi,\xi'\in\Z}\sum_{m,x,x'\in\T_n}e^{-2\pi i(m+\xi+\eta)\frac xn}e^{2\pi i (m+\xi')\frac {x'}n}\; 
(\mc F  r)^\star(\xi)
(\mc F  r)(\xi') ({\cal F}G)^\star\Big(\eta,\frac{m}{n}\Big).
$$
Due to  smoothness of $r(\cdot)$, its Fourier coefficients decay rapidly. 
Therefore for a fixed $\eta \in \bb Z$ and for any $\rho\in(0,1)$  we have that the limit 
$\lim_{n\to+\infty} f_{n}(\eta)$ equals
$$
\lim_{n\to+\infty}\frac{1}{2n^2}\sum_{\substack{|\xi|\le n^{\rho}\\|\xi'|\le n^{\rho}}}\sum_{m,x,x'\in\T_n}e^{-2\pi i(m+\xi+\eta)\frac xn}e^{2\pi i (m+\xi')\frac{x'}n}\; 
(\mc F  r)^\star(\xi)(\mc F  r)(\xi') ({\cal F}G)^\star\Big(\eta,\frac{m}{n}\Big).
$$
Summing over $x,x'$ we conclude that $ \lim_{n\to+\infty} f_{n}(\eta)$ equals
\[
\lim_{n\to+\infty} \frac{1}{2}\sum_{\substack{|\xi|\le n^{\rho}\\|\xi'|\le n^{\rho}}}
\sum_{m\in\T_n}1_{\Z}\Big(\frac{m+\xi+\eta}{n}\Big)1_{\Z}\Big(\frac{m+\xi'}{n}\Big)
(\mc F  r)^\star(\xi)(\mc F  r_0)(\xi') ({\cal F}G)^\star\Big(\eta,\frac{m}{n}\Big).
\]
Taking into account the fact that $m\in\T_n$ and the magnitude of $\xi,\xi'$ is negligible when compared with $n$ we conclude that the terms under the summation on the right hand side are non zero only if $m+\xi'=0$ and  $m+\xi+\eta=0$, or $m+\xi'=n$ and  $m+\xi+\eta=n$. Therefore,
\begin{align*}
\lim_{n\to+\infty} f_{n}(\eta) & =
\lim_{n\to+\infty} \frac{1}{2}\sum_{\substack{|\xi|\le n^{\rho}\\|\xi'|\le n^{\rho}}}
\sum_{m\in\T_n}\delta_{n,m+\xi+\eta}\;\delta_{n,m+\xi'}\; (\mc F  r)^\star(\xi)
(\mc F  r)(\xi') ({\cal F}G)^\star\left(\eta,\frac{m}{n}\right)\\
& \quad 
+\lim_{n\to+\infty} \frac{1}{2}\sum_{\substack{|\xi|\le n^{\rho}\\|\xi'|\le n^{\rho}}}
\sum_{m\in\T_n}\delta_{0,m+\xi+\eta}\;\delta_{0,m+\xi'}\; (\mc F  r)^\star(\xi)
(\mc F  r)(\xi') ({\cal F}G)^\star\left(\eta,\frac{m}{n}\right)
\\
&
=\frac{1}{2}\; ({\cal F}G)^\star(\eta,0)\sum_{\xi\in\Z}(\mc F  r)^\star(\xi)
(\mc F  r)(\xi+\eta)
\end{align*}
and formula \eqref{050704b} follows.
\end{proof}

\bigskip

In particular, for the initial conditions of our dynamics, Proposition \ref{propo021004} and \eqref{eq:allequal} imply:
\begin{equation}
\label{050704}
\begin{split}
  \lim_{n\to+\infty}\sum_{\eta\in\Z}\Big[\overline{W}_n^+(\eta,\cdot)({\cal F}G)^\star(\eta,\cdot)\Big]_n
&=\sum_{\xi,\eta\in\Z}W(r_0\; ; \; \eta,\xi)({\cal F}G)^\star(\eta,0) \\
&  = {\frac12} \int_\T r_0^2(u) {G^\star}(u,0) \; \dd u.
\end{split}
\end{equation}
One of the main point of the proof of our theorem is to show that this
convergence holds for  any macroscopic time $t>0$, i.e.
that for any compactly supported $G\in {\cal C}^\infty(\R_+\times\T^2)$ we have
  \begin{equation}
\begin{split}   \lim_{n\to+\infty}\sum_{\eta\in\Z}\int_0^{+\infty}\Big[\overline{W}_n^+(t, \eta,\cdot)&({\cal F}G)^\star(t,\eta,\cdot)\Big]_n\dd t\\
& =\sum_{\xi,\eta\in\Z}\int_0^{+\infty}W(r_t\; ; \; \eta,\xi)({\cal F}G)^\star(t,\eta,0)\dd t\\ & = \frac12\int_{\R_+\times \T} r^2(t,u) G(t,u) \; \dd t\dd u.\end{split} \label{eq:quest}
\end{equation}
This would amount to showing that the Wigner distribution $\overline{W}_n^+(t)$ associated to $\overline\psi_x(t)$ 
is asymptotically equivalent to 
the one corresponding  to the macroscopic profile $r\left(t, \cdot\right)$ via
\eqref{040704}. 
This fact is not a consequence of Theorem \ref{theo:hydro}, that implies only 
a weak convergence of  $\overline\psi_x(t)$ to $r\left(t, \cdot\right)$.
We will prove \eqref{eq:quest} in Proposition \ref{lem:mech} below, showing  the convergence of  the
 corresponding Laplace transforms (see \eqref{011207}).



\section{Strategy of the proof and explicit resolutions}
\label{sec:strategy}

\subsection{Laplace transform of Wigner functions} \label{sec:lap1}
Since our subsequent  argument is based on an application of the Laplace
transform of the Wigner functions, we give some  explicit formulas
for the object
that  can be written in case of our model. For any  bounded
complex-valued, Borel measurable function $ \R_+\ni t \mapsto f_t$ we define the Laplace operator $\mc L$ as:
\begin{equation*}
\mc L(f_\cdot)(\lambda):=\int_{0}^{+\infty} e^{-\lambda t} f_t\; \dd t, \qquad \lambda >0.
\end{equation*} 
Given the solution $r_t = r(t,\cdot)$ of \eqref{eq:elong-evol}, we  define the Laplace transform of the Wigner distribution \eqref{040704} associated to a macroscopic profile, as follows: for any $(\lambda,\eta,\xi) \in \R_+\times \bb Z^2$, 
\begin{equation*}
  w(r_\cdot\; ;\;  \eta, \xi)(\lambda)=\mathcal L\big(W(r_\cdot\; ; \; \eta,\xi)\big)(\lambda)  = \int_0^{+\infty} e^{-\lambda t} \; W(r_t\; ;\; \eta,\xi) \; \dd t.
\end{equation*}
The following formulas are easily deduced, by a direct calculation, from \eqref{eq:elong-evol}, and are left to the reader:
\begin{lemma}
  \label{lemma-laplace}
  For any $(\lambda,\eta,\xi) \in \R_+\times \bb Z^2$ we have
  \begin{equation}
\label{011007}
 w(r_\cdot \; ;\;   \eta, \xi)(\lambda) = \frac{W(r_0\; ;\; \eta,\xi)}{\frac {2\pi^2}{\gamma}\left[\xi^2 + (\eta+\xi)^2\right] + \lambda}.
\end{equation}
Consequently,
\begin{align}
 \frac12 \mathcal L \Big({\cal F} \left((\partial_u r_t)^2\right)(\eta)\Big)(\lambda)& = \frac12\int_0^{+\infty} e^{-\lambda t} \; \mc F\left((\partial_u r_t)^2\right)(\eta) \; \dd t \notag\\ & =
    4\pi^2 \sum_{\xi\in\Z} (\eta+\xi) \xi \; w(r_\cdot \; ;\;   \eta, \xi)(\lambda). \vphantom{\int_0^{+\infty}} \label{eq:s4}
\end{align}
\end{lemma}

Finally, we define
$\mb w_n^+$  the Laplace transform of $\mb W_n^+$ as 
the tempered distribution given for any $ G \in \mc C^\infty(\T\times\T)$ and $\lambda >0$ by  
\begin{equation}
\big\langle \mb w^+_n(\la),\; G \big\rangle=\int_0^{+\infty}e^{-\la
  t}\; \big\langle \mb W^+_n(t),\; G \big\rangle \,\dd t
:= \frac1n  \sum_{k\in\widehat{\T}_n}\sum_{\eta\in\Z} w^+_n(\la,\eta,k) (\mc FG)^\star(\eta,k), \label{eq:def-laplace} 
\end{equation}
where
$w_n^+$ is  the Laplace transform of $W_n^+$  as follows:
\[
{w}^+_n(\lambda, \eta,k):= \int_0^{+\infty} e^{-\lambda t} \;
W_n^+(t,\eta,k)\; \dd   t, \qquad \quad
(\lambda,\eta,k)\in\R_+\times \Z\times
\widehat{\T}_n.
\] 
In a similar fashion we can also define  $\mb w_n^-(\la)$ and $\mb
y_n^\pm(\la)$ the Laplace transforms of $\mb W_n^-(t)$ and $\mb
Y_n^\pm(t)$, respectively, and their counterparts $w_n^-$, and $y_n^\pm$.

\subsection{Dynamics of the Wigner distributions}
\label{sec:time}

Using the time evolution equations \eqref{eq:dynamics-psi}, one can first write a closed system of evolution equations for $W_n^\pm(t),  Y_n^\pm(t)$ defined respectively in \eqref{010803}, \eqref{eq:w-}, \eqref{eq:y+}, \eqref{eq:y-}. 

For that purpose we first define
two functions $\delta_n s$ and $\sigma_ns$   as follows: for any $(\eta,k)\in\Z\times\widehat \T_n$,
\begin{align}
\label{012308}
(\delta_n s)(\eta,k)&:=2n\Big(\sin^2\left(\pi\Big(k+\frac{\eta}{n}\Big)\right)-\sin^2(\pi k)\Big), \vphantom{\Bigg\{}
\\
(\sigma_ns)(\eta,k)&:=2\Big(\sin^2\left(\pi\Big(k+\frac{\eta}{n}\Big)\right)+\sin^2(\pi k)\Big).\nonumber
 \end{align}
 For the brevity sake, we drop the variables $(t,\eta,k) \in
 \R_+\times\Z \times \widehat \T_n$ from the subsequent
 notation. From \eqref{eq:dynamics-psi} one can easily check that:
\begin{equation}\label{eq:closed-syst} \left\{\begin{aligned}
\partial_t W^+_n& = -in  (\delta_ns)\;  W_n^+ -n^2\sin(2\pi k)\; Y_n^+-n^2\sin\big(2\pi\big(k+\tfrac{\eta}{n}\big)\big)\; Y_n^-  \\
&
 \quad + \gamma n^2\; \bb L\big(2W_n^+-{Y^+_n-Y_n^-}\big),\vphantom{\sum_{k\in\widehat{\T}_n}} \\
\partial_t Y^+_n & =n^2\sin(2\pi k)\; W_n^+-i n^2 (\sigma_n s) \; Y^+_n -n^2\sin\big(2\pi\big(k+\tfrac{\eta}{n}\big)\big)\; W_n^-\\
&
\quad + \gamma n^2\;\bb L\big(2Y^+_n - {W_n^+-W^-_n}\big)+ \gamma n\sum_{k\in\widehat{\T}_n} (Y_n^--Y_n^+),\\
\partial_t  Y^-_n & = n^2\sin\big(2\pi\big(k+\tfrac{\eta}{n}\big)\big)\; W_n^+ +i n^2 (\sigma_n s) \; Y^-_n -n^2\sin(2\pi k )\; W_n^-\\
& \quad +\gamma n^2 \;\bb L\big(2Y^-_n - {W_n^+-W^-_n}\big)+\gamma n \sum_{k\in\widehat{\T}_n} (Y_n^+-Y_n^-),\\
\partial_t W^-_n & = in(\delta_ns)\; W^-_n +n^2\sin\big(2\pi\big(k+\tfrac{\eta}{n}\big)\big)\;Y_n^++n^2\sin(2\pi k)\; Y_n^-\\
&
\quad + \gamma n^2 \;\bb L\big(2W^-_n - {Y^+_n-Y_n^-}\big), 
\end{aligned} \right.\end{equation}
where  $\bb L$ is the operator that is defined for any $f:\Z\times\widehat{\T}_n \to \C$ as 
\[
(\bb L f)(\eta,k):=\big[f(\eta,\cdot)\big]_n -
f(\eta,k),\qquad (\eta,k)\in \Z\times\widehat{\T}_n. 
\]
Recalling the decomposition \eqref{eq:decompo} and the evolution equations \eqref{eq:dynamics-overlinepsi} for the mean part of the wave function, 
we have similarly that $\overline W_n^\pm(t), \overline Y_n^\pm(t)$ 
satisfy the autonomous equations:
\begin{equation}\label{eq:closed-overline} \left\{\begin{aligned}
\partial_t \overline W^+_n& = -in  (\delta_ns)\; \overline W_n^+ -n^2\sin(2\pi k)\; \overline Y_n^+
-n^2\sin\big(2\pi\big(k+\tfrac{\eta}{n}\big)\big)\; \overline Y_n^- \\
&
 \quad - \gamma n^2 \big(2 \overline W_n^+-{\overline Y^+_n-\overline Y_n^-}\big), \vphantom{\Big(}\\
\partial_t \overline Y^+_n & =n^2\sin(2\pi k)\; \overline W_n^+-i n^2 (\sigma_n s) \; \overline Y^+_n 
-n^2\sin\big(2\pi\big(k+\tfrac{\eta}{n}\big)\big)\; \overline W_n^-\\
&
\quad - \gamma n^2\; \big(2 \overline Y^+_n - {\overline W_n^+-\overline W^-_n}\big), \vphantom{\Big(}\\
\partial_t  \overline Y^-_n & = n^2\sin\big(2\pi\big(k+\tfrac{\eta}{n}\big)\big)\; \overline W_n^+ +i n^2 (\sigma_n s) \;\overline Y^-_n 
-n^2\sin(2\pi k )\; \overline W_n^-\\
& \quad - \gamma n^2 \;\big(2\overline Y^-_n - {\overline W_n^+-\overline W^-_n}\big)\vphantom{\Big(}\\
\partial_t \overline W^-_n & = in(\delta_ns)\; \overline W^-_n +n^2\sin\big(2\pi\big(k+\tfrac{\eta}{n}\big)\big)\;\overline Y_n^+
+n^2\sin(2\pi k)\; \overline Y_n^-\\
&
\quad - \gamma n^2 \;\big(2\overline W^-_n - {\overline Y^+_n-\overline Y_n^-}\big), 
\end{aligned} \right.\end{equation}

\subsection{Laplace transform of  the dynamical system}
\label{sec:lap}

%

We deduce from \eqref{eq:closed-syst} an equation satisfied by 
$\mf w_n$ -- the four-dimensional vector of Laplace transforms of the Wigner functions defined by
$ \mf w_n:= [w_n^+,y_n^+,y_n^-,w_n^-]^{\rm T}.$
For the clarity  sake we  shall use the notation
\begin{align*}
{\mb 1} &: = [1,1,1,1]^{\rm T}, \nonumber \vphantom{\big\{} \quad
\rm e  :=[1,-1,-1,1]^{\rm T}, \nonumber \vphantom{\Big\{}\\ 
\mf v_n^0 & := \big[ W_n^+(0),
  Y_n^+(0),
 Y_n^-(0) ,
 W_n^-(0)\big]^{\rm T},\vphantom{\Big\{} \\ 
\overline{\mf v}_n^0 & := \big[ \overline W_n^+,
 \overline  Y_n^+,
 \overline Y_n^- ,
\overline  W_n^-\big]^{\rm T},\vphantom{\Big\{}
\qquad\widetilde{\mf v}_n^0  := \big[ \widetilde W_n^+,
\widetilde   Y_n^+,
\widetilde  Y_n^- ,
 \widetilde W_n^-\big]^{\rm T},\vphantom{\Big\{} \end{align*}
 and from \eqref{eq:decompo} we have $\mf v_n^0=\overline{\mf v}_n^0 + \widetilde{\mf v}_n^0$. We also remark that \eqref{eq:sumequal} implies: 
\[ \Big[ w_n^+(\lambda,\eta,\cdot)\Big]_n = \Big[ w_n^-(\lambda,\eta,\cdot)\Big]_n, \qquad \text{for any } (\lambda,\eta)\in\bb R_+ \times \bb Z.\]
Let us finally define $\mc I_n$ as the scalar product  
 \begin{equation*}
 \mc I_n(\lambda,\eta) := {\rm e} \cdot\Big[ {\mf w}_n(\lambda,\eta,\cdot)\Big]_n  = \Big[ \big(w_n^+-y_n^+-y_n^- + w_n^-\big)(\lambda,\eta,\cdot)\Big]_n.
\end{equation*}
As before, we shall often drop the
 variables $(\lambda,\eta,k)$ from the notations. 
 We are  now ready to take the Laplace transform of both sides of  \eqref{eq:closed-syst}:  we obtain a linear system that can be written for any $(\lambda,\eta,k)$ in the form
\begin{equation}
\label{eq:linear-syst}
(\mb M_n \; \mf w_n)(\lambda,\eta,k) = \mf v_n^0(\eta,k) + \gamma
n^2\;\mc I_n(\lambda,\eta)\; {\rm e},
\end{equation}
where the $2\times 2$ block matrix $\mb M_n:=\mb M_n(\lambda,\eta,k)$
is defined as follows:
\begin{equation}
\label{m-n}
\mb M_n:=\begin{bmatrix}
A_n & -n^2\; \gamma_n^{-} \;  {\rm Id}_2 \vphantom{\Big(}\\
-n^2\; \gamma_n^{+}\;  {\rm Id}_2& B_n \vphantom{\Big(}
\end{bmatrix},
\end{equation}
where, given a positive integer $N$,   $ {\rm Id}_N$ denotes the 
$N\times N$ identity matrix, and $A_n$, $B_n$ are $2\times 2$ matrices:
\[
A_n:=\begin{bmatrix}
a_n & -n^2\gamma^-  \vphantom{\Big(}\\
-n^2\gamma^+  & b_n \vphantom{\Big(}
\end{bmatrix},\qquad  B_n:=\begin{bmatrix}b_n^\star &-n^2\gamma^- 
  \vphantom{\Big(}\\
 -n^2\gamma^+ & a_n^\star \vphantom{\Big(}
\end{bmatrix}.
\]
Here and below,
\begin{equation}
\label{a-n}
\left\{\begin{aligned} a_n&:=\lambda + in(\delta_n s) + 2\gamma n^2,\\
b_n&:=\lambda + in^2(\sigma_n s) + 2\gamma n^2,\end{aligned}\right. \qquad \left\{\begin{aligned} 
\ga_n^\pm&:=\ga\pm\sin\big(2\pi\big(k+\tfrac \eta n\big)\big), \\ 
\ga^\pm&:=\ga\pm\sin(2\pi k).\end{aligned}\right.
\end{equation}
%
An elementary observation yields the following symmetry properties
\begin{equation}
\label{041307}
\ga_n^\pm(-\eta,-k)=\ga_n^\mp(\eta,k),\quad \ga^\pm(-\eta,-k)=\ga^\mp(\eta,k),
\end{equation}
By the linearity of the Laplace transform we can write 
\begin{equation} \label{eq:dec-def}{\mf w}_n=\overline{\mf w}_n + \widetilde{\mf w}_n, 
\end{equation} 
where $\overline{\mf w}_n$ is the Laplace transform of $(\overline W^+_n(t), \overline Y^+_n(t), \overline Y^-_n(t), \overline W^-_n(t))$, and
$\widetilde{\mf w}_n$ is the Laplace transform of $(\widetilde W^+_n(t), \widetilde Y^+_n(t), \widetilde Y^-_n(t), \widetilde W^-_n(t))$.

{Performing the Laplace transform of both sides of  \eqref{eq:closed-overline} 
we conclude that $\overline{\mf w}_n$ solves the equation
\begin{equation}\label{eq:mech-laplace}
{\mb M}_n \overline{\mf w}_n = \overline{\mf v}_n^0 = \overline{W}^+_n\; \mathbf 1 .
\end{equation}
 Using \eqref{eq:linear-syst} we conclude that
 $\widetilde{\mf w}_n$ solves}
\begin{equation}\label{eq:term-laplace}
  {\mb M}_n \widetilde{\mf w}_n = \widetilde{\mf v}_n^0 + \gamma n^2 \; {\mc I}_n \; {\rm e}. \vphantom{\Big(}
\end{equation}
Following \eqref{eq:dec-def} we also write $ {\mc I}_n =  \overline{\mc I}_n + \widetilde{\mc I}_n$, where
\begin{equation}
  \label{eq:Idec}
 \overline{\mc I}_n (\lambda,\eta) :=  {\rm e} \cdot
  \Big[ \overline{\mf
    w}_n(\lambda,\eta,\cdot)\Big]_n \quad \text{ and } \quad
 \widetilde{\mc I}_n(\lambda,\eta):= {\rm e} \cdot   \Big[\widetilde{\mf w}_n(\lambda,\eta,\cdot)  \Big]_n.\nonumber
\end{equation}
In Section \ref{app:inversion}, we show that the matrix ${\mb M}_n$ is invertible, therefore we can solve and rewrite \eqref{eq:mech-laplace} and \eqref{eq:term-laplace} as:
\begin{align}
  \label{eq:inverse-s1}
     \overline{\mf w}_n &= \overline{W}^+_n \;{\mb M}_n^{-1}\; \mathbf 1 , \vphantom{\Big\{}\\
  \label{eq:inverse-s2}
     \widetilde{\mf w}_n &=  {\mb M}_n^{-1}\; \widetilde{\mf v}_n^0 + \gamma n^2\; {\mc I}_n \; {\mb M}_n^{-1} \;{\rm e}. 
\end{align}
In Section \ref{sec:proof} we  study the contribution of
the terms appearing in the right hand sides of both
\eqref{eq:inverse-s1} and \eqref{eq:inverse-s2} that reflect upon the
evolution of the mechanical and fluctuating components  of the energy functional. 

\section{Proof of the hydrodynamic behavior of the energy}

In this section we conclude the proof of Theorem \ref{theo:hydro1}, up to technical lemmas that are proved in Section \ref{app:proof}.

\label{sec:proof}

\subsection{Mechanical energy $\overline{\mf w}_n$}
\label{sec:results-depend-only}





We start with the recollection of  the results concerning  the mechanical energy. 
The Laplace transform
$\overline{\mf w}_n$ is autonomous from the thermal part and satisfies \eqref{eq:inverse-s1}.
Let us introduce, for any $\lambda >0$ and $\eta \in \Z$, the \emph{mechanical Laplace-Wigner function}
\begin{equation}
  \label{eq:mech-en-macro}
\mc W_{\rm  mech}^+(\lambda,\eta):=\sum_{\xi\in\Z}\frac{
W(r_0\; ; \;\eta,\xi)}{\frac{2\pi^2}{\ga}[\xi^2+(\xi+\eta)^2]+\la}.
\end{equation}
From Lemma \ref{lemma-laplace}, it follows that $\mc W_{\rm  mech}^+(\lambda,\eta)$ is the Fourier-Laplace transform
of the mechanical energy density $e_{\rm mech}(t,u)=\frac12 \left(r(t,u)\right)^2$,
where $r(t,u)$ is the solution of \eqref{eq:elong-evol}. 

Given $M\in\bb N$ we denote  by ${\cal P}_M$ the subspace of ${\cal
    C}^\infty(\T\times \T)$ consisting of all trigonometric
  polynomials that are   finite linear combinations of  $e^{2\pi i\eta 
    u}e^{2\pi i\xi v}$, with $\eta\in\{-M,\ldots,M\}$, $\xi\in\bb Z$ and $u,v\in\T$.
    
\begin{proposition}[Mechanical part] \label{lem:mech}  For any
  $M\in\bb N$ there exists
  $\lambda_M >0$ such that for any $G\in{\cal P}_M$ and $\la>\la_M$ we have
\begin{equation}
\label{011207}
\begin{split}
  \lim_{n\to\infty} \sum_{\eta\in\Z} \Big[ \overline{\mf
    w}_n(\lambda,\eta,\cdot) (\mc F G)^\star(\eta,\cdot) \Big]_n\; & =
  {\bigg(}\sum_{\eta, \xi \in\Z} \frac{W(r_0\; ;
    \;\eta,\xi)}{\frac{2\pi^2}{\ga}[\xi^2+(\xi+\eta)^2]+\la} (\mc F
  G)^\star(\eta,0) {\bigg)} {\mb 1}\\
   & = {\bigg(\frac12}\int_\T \int_0^{+\infty} e^{-\lambda t} \; \left(r(t,u)\right)^2\; {G^\star}(u,0) \; \dd t\; \dd u{\bigg) \mathbf{1}}. 
\end{split}
\end{equation}
Moreover, for any $\eta\in\{-M,...,M\}$ 
\begin{align}
\lim_{n\to\infty}\Big\{\gamma n^2 \; \overline{\mc I}_n (\lambda,\eta) \Big\} & = 
  \frac{ 4\pi^2}{\gamma} \sum_{\xi\in\Z} \frac{\xi(\xi+\eta)W(r_0\; ; \; \eta,\xi)}{\frac{2\pi^2}\ga[\xi^2+(\xi+\eta)^2]+\la} \notag\\ & = \frac{1}{2\gamma} \mathcal L
\Big({\cal F} \left((\partial_u
  r)^2\right)(\eta)\Big)(\lambda).
\label{eq:nonlin1}
\end{align}

\end{proposition}
The proof of Proposition \ref{lem:mech} is exposed in Section \ref{sec:proof-proposition-mech}.

\subsection{The closing of thermal energy equation}
\label{sec:clos-therm-energy}

We now analyse equation \eqref{eq:inverse-s2} concerning the fluctuating part. 
After averaging \eqref{eq:inverse-s2} over $k\in\widehat{\T}_n$ and scalarly multiplying by ${\rm e}$,
 we obtain the  equation:
\begin{equation} \label{eq:equa-itilde}
\widetilde{\mc I}_n= \widetilde{z}_n^{(0)}  + \gamma n^2\; \big(\widetilde{\mc I}_n + \overline{\mc I}_n\big)\;  {\mc M}_n ,\end{equation}
where 
\begin{align*}
 \vphantom{\bigg(}\widetilde{z}_n^{(0)} &:=\Big[{\rm e}\cdot ({\bf M}_n^{-1}\;\widetilde{\mf v}_n^0)\Big]_n
  \\ 
{\mc M}_n&:= \Big[{\rm e}\cdot ({\bf M}_n^{-1}\;{\rm e})\Big]_n.
\end{align*}
Therefore from \eqref{eq:equa-itilde} we solve explicitly 
\begin{equation}
  \label{eq:ictilde}
  \widetilde{\mc I}_n (\lambda, \eta) = 
  \frac{n^2 \; \widetilde{z}_n^{(0)} (\lambda, \eta)  + \big(\gamma n^2\; \overline{\mc I}_n  (\lambda, \eta)\big)\; 
  \big( n^2{\mc M}_n  (\lambda, \eta)\big)}{n^2\big(1-\gamma n^2\; {\mc M}_n  (\lambda, \eta)\big)}.
\end{equation}
Asymptotics of $n^2\overline{\mc I}_n  (\lambda, \eta)$ is given in \eqref{eq:nonlin1}. Below we describe the   terms $n^2 \; \widetilde{z}_n^{(0)} (\lambda, \eta) $  and $n^2{\mc M}_n  (\lambda, \eta)$ that also appear in the right hand side of \eqref{eq:ictilde}.
\begin{lemma}
Fix $M\in\N$. There   exists $\lambda_M>0$ such that, for any $\lambda
>\lambda_M$ and  $\eta\in\{-M,...,M\}$ 
\label{lemma2}
  \begin{align}
&    
   \lim_{n\to\infty} \Big\{ \gamma n^2\; {\mc M}_n(\lambda,\eta) \Big\}= 1, 
 \label{eq:1st} \\
    \label{eq:2nd}
  &   \lim_{n\to\infty} \Big\{n^2\; \big(1 - \gamma n^2\; {\mc M}_n(\lambda,\eta) \big)\Big\}  
    = \frac 1{2\gamma} \Big(\lambda + \frac{\eta^2\pi^2}{\gamma}\Big), \vphantom{\Bigg\{} 
  \end{align}
and
  \begin{equation}
    \label{eq:4th}
    \lim_{n\to\infty}\Big\{ \gamma n^2\; \vphantom{\bigg(}\widetilde{z}_n^{(0)} (\lambda,\eta) \Big\}
 =  \big(\mathcal Fe_{\rm{thm}}(0,\cdot)\big)(\eta).
  \end{equation}
\end{lemma}

As a direct consequence of the above lemma and \eqref{eq:nonlin1}, we  obtain:
\begin{corollary}\label{cor:conv} Fix $M\in\N$. There exists $\lambda_M>0$ such that, for any $\lambda
>\lambda_M$,  $\eta\in\{-M,...,M\}$
\begin{equation*}
   \lim_{n\to\infty} \widetilde{\mc I}_n (\lambda,\eta) =2\mc W_{\rm thm}^+(\lambda,\eta),
\end{equation*}
where 
\begin{equation*}\mc W_{\rm thm}^+(\lambda,\eta)
:=\Big(\la+\frac{\eta^2\pi^2}{\ga}\Big)^{-1}\bigg\{\vphantom{\int_0^1}\big({\cal
  F}e_{\rm thm}(0,\cdot)\big)(\eta)+\frac{1}{2\gamma} \mathcal L
 \Big({\cal F} \left((\partial_u
  r)^2\right)(\eta)\Big)(\lambda)\bigg\}.\end{equation*}
\end{corollary}

The following lemma finalizes the identification of the limit for the Fourier transform of the thermal energy:
\begin{proposition}\label{lemma3}
  Fix $M\in\N$. There exists $\lambda_M>0$ such that, for any $\lambda
>\lambda_M$,  $\eta\in\{-M,...,M\}$ \begin{equation}
    \label{eq:3}
    \lim_{n\to\infty} \left\{\widetilde{\mc I}_n (\lambda,\eta) - 2 \; \big[\widetilde{ w}^+_n(\lambda,\eta,\cdot)\big]_n\right\} = 0.
  \end{equation}
\end{proposition}
The proofs of Lemma \ref{lemma2} and Proposition \ref{lemma3} go very much along the lines of  the arguments presented in Section \ref{app:proof} and we will not present the details here.
They are basically  consequences of the following limit
\begin{equation*}
  \lim_{n\to\infty} \big\{n^2 \; {\rm e_1} \cdot {\bf M}_n^{-1} (\lambda, \eta,k)\; {\rm e}\big\} = \frac 1{2\gamma},
\end{equation*}
which is proved in Section \ref{ssec:asymp}.

\subsection{Asymptotics of  $\widetilde{\mf w}_n$ and $\mf w_n$}
\label{sec:asym}

With a little more work one can prove the following \emph{local equilibrium} result, which is an easy consequence of Proposition \ref{lem:mech}, Corollary \ref{cor:conv} and Proposition \ref{lemma3} (recall also \eqref{eq:inverse-s2}). 

\begin{theorem}\label{theo:limitW}
Fix $M\in\bb N$.  There exists $\lambda_M>0$ such that, for any $\lambda
>\lambda_M$   and $G\in {\cal P}_M$ we have 
\begin{multline}
\label{101407}
\lim_{n\to+\infty}\sum_{\eta\in\bb Z}
\Big[w_n^+(\lambda, \eta,\cdot) ({\cal F}G)^\star(\eta,\cdot)\Big]_n \\=\sum_{\eta\in\bb Z}\bigg\{\mc W_{\rm thm}^+(\lambda,\eta)\; \int_\T  ({\cal F}G)^\star(\eta,v)\; \dd v+\mc
W_{\rm mech}^+(\lambda,\eta)\;  ({\cal F}G)^\star(\eta,0)\bigg\}
\end{multline}
and 
\begin{equation}
\label{111407}
\lim_{n\to+\infty}\sum_{\eta\in\bb Z}
\Big[y_n^+(\lambda, \eta,\cdot) ({\cal F}G)^\star(\eta,
\cdot) \Big]_n =\sum_{\eta\in\bb Z} \mc
W_{\rm mech}^+(\lambda,\eta)\;  ({\cal F}G)^\star(\eta,0),
\end{equation}
\end{theorem}
{We will not give the details for the proof of this last theorem, since the argument is very similar to Proposition \ref{lem:mech}.} 

\subsection{End of the proof of Theorem \ref{theo:hydro1}} \label{sec:end}

The proof of convergence \eqref{eq:convergence} has been reduced to the investigation of the Wigner distributions.
Recall that from the uniform bound \eqref{eq:bound}, we know that the sequence of all
Wigner distributions $\{\mb W_{n}^+(\cdot), \mb Y_{n}^+(\cdot), \mb Y_{n}^-(\cdot), \mb W_{n}^-(\cdot)\}_{n}$ is
sequentially pre-compact with respect to the  $\star$-weak
topology  in the dual space of ${\bf L}^1(\R_+,\mc A_0)$. More precisely, one can choose a subsequence $n_m$ such that 
any of the components above, say  for instance $\mb W_{n_m}^+(\cdot)$, $\star$-weakly converges
 in the dual space of ${\bf L}^1(\R_+,\mc A_0)$ to some  $\mb W^+(\cdot)$. 

To characterize its limit, we  consider
$\mb w_{n_m}^+(\lambda)$
obtained by taking the Laplace transforms of the respective $\mb W_{n_m}^+(\cdot)$. 
For any $\la>0$, it converges  $\star$-weakly, as $n_m\to+\infty$,  
in ${\cal A}_0'$ to some $\mb w^+(\lambda)$ that is the Laplace transform of $\mb W^+(\cdot)$.
The latter is defined as 
$$
\langle \mb w^+(\lambda),G\rangle:=  
\int_0^\infty \langle\mb W^+(t), e^{-\la t} G\rangle\; \dd t  \qquad \la>0, \; G\in {\cal A}_0.
$$
Given a trigonometric polynomial $G\in {\cal C}^\infty(\T\times\T)$ we conclude, thanks to Theorem \ref{theo:limitW}, that for any $\la>\la_M,$
\begin{equation}
\label{012303}
\left\langle\mb  w^+(\lambda),G\right\rangle=\int_{\R_+\times\T^2}e^{-\la t}\; e(t,u) G(u,v)\; \dd t\; \dd u\; \dd v,
\end{equation}
where $e(t,u)$ is defined as in Theorem \ref{theo:hydro1} and $M\in\bb
N$ is such that ${\cal F}G(\eta,v)\equiv 0$ for all $|\eta|>M$. 

Due to the uniqueness of the Laplace transform (that can be argued by analytic continuation), this proves that in fact
 equality \eqref{012303} holds for all $\la>0$. By a density argument it can be then extended to all $G\in{\cal A}_0$ and shows that
$\mb W^+(t,u,v)=e(t,u)$, for any  $(t,u,v)\in \R_+\times\T^2$. This ends the proof of \eqref{eq:convergence}, and thus Theorem \ref{theo:hydro1}.

\section{Proofs of the technical results stated in Section \ref{sec:proof}}
\label{app:proof}

In what follows we shall adopt the following notation: 
we say that the sequence $C_n(\la,\eta,k) \preceq 1$  if  for any given integer $M \in \mathbb N$, 
there exist $\lambda_M>0$ and $n_M\in\N$ such that
 $$
\sup\Big\{ C_n(\lambda,\eta,k)\; ; \,\lambda >
  \lambda_M, \,\eta \in \{-M,...,M\},\,n>n_M , \,k\in\widehat{\T}_n\Big\}<+\infty.
$$

\subsection{Invertibility of $\mb M_n(\lambda,\eta,k)$}
\label{app:inversion}

\begin{proposition}
The matrix $\mb M_n(\lambda,\eta,k)$ defined in \eqref{m-n} is invertible for all
$n\ge 1$, $\la>0$ and $(\eta,k)\in\Z\times \widehat{\T}_n$.
\end{proposition}

\begin{proof}
The block entries of the matrix ${\mb M}_n$ defined in \eqref{a-n} satisfy the commutation relation
$$
[A_n,B_n]=A_nB_n-B_nA_n=\begin{bmatrix}
0& -2\gamma^- n^2\; {\rm Re}[a_n-b_n] \vphantom{\Big(}\\
-2\gamma^+ n^2\; {\rm Re}[ b_n-a_n]  &  0\vphantom{\Big(}
\end{bmatrix}=0.
$$ 
Thanks to the well known formula for the determinants of  block matrices with commuting entries we have (see e.g.   \cite[formula (Ib), p. 46]{gant})
\[\text{det}(\mb M_n)= 
\text{det}(A_nB_n-\ga_n^+\ga_n^- n^4\; {\rm Id}_2)=\left|a_nb_n^\star+n^3(\delta\ga_n)\right|^2
-4n^4\ga^+\ga^- ( {\rm Re}[a_n])^2\notag 
\]
and, substituting from \eqref{a-n}, we get  \begin{align} 
\text{det}(\mb M_n) = & \; n^6 \; \Big( (\delta_n s)\; (\sigma_n
s) + (\delta\ga_n)\Big)^2 \notag \\
&
 + \big(\lambda + 2\gamma n^2\big)^2  \; \bigg\{\lambda^2+n^2\big(4\gamma\lambda  +(\delta_n s)^2\big) \notag  \\ & \qquad   +n^4 \Big(
 2 \sin^2(2\pi
    k) +2\sin^2\big(2\pi\big(k+\tfrac{\eta}{n}\big)\big)+ (\sigma_n s)^2\Big)
\bigg\}.\label{031203}
\end{align}
Here  $\delta_n s$, $\sigma_n
s$ are given by  \eqref{012308} and
$$
\delta\ga_n(\eta,k):=n(\ga^+\ga^--\ga^+_n\ga^-_n)=n\Big(\sin^2\big(2\pi\big(k+\tfrac{\eta}{n}\big)\big)-\sin^2(2\pi
k)\Big).
$$
The proposition is a direct conclusion of \eqref{031203}. 
\end{proof}

\bigskip

It is also  clear that
\begin{equation}
\label{032308}
\text{det}(\mb M_n)=n^8\Delta_n,
\end{equation}
 where
\begin{equation}
\Delta_n= \frac 1{n^2} \bigg\{4\ga^2\; \Gamma_n + 4\ga^2\big(4\la \ga+(\delta_ns)^2\big) 
+ \big( (\delta_n s)(\sigma_n s)+
(\delta\ga_n)\big)^2 \bigg\} + 4\gamma\lambda\;\frac{ \Gamma_n}{n^4}  +\frac{C_n}{n^3}.  \label{051403}
\end{equation}
for some $|C_n|\preceq 1$ and
\begin{equation}
\label{Gan}
\Gamma_n(\eta,k):=n^2\Big(2\sin^2(2\pi k) +2\sin^2\big(2\pi
\big(k+\tfrac{\eta}{n}\big)\big)+(\si_n s)^2\Big).
\end{equation}
{On the one hand, note that for $k$ sufficiently \emph{far} from $0$, the dominant term is $(4\gamma^2\; \Gamma_n)/n^2$ and then
$$
\Delta_n \sim 4\gamma^2 \left(4\sin^2(2\pi k) + 16\sin^4(\pi k) \right).
$$
On the other hand, for $k = \frac{\xi}{n}$ and fixed $\xi\in\Z$ we have
\begin{equation}
  \label{eq:1}
  n^2\; \Delta_n\Big(\lambda, \eta, \frac{\xi}{n}\Big) =\frac 1{n^6} \text{det}(\mb M_n)\Big(\lambda, \eta, \frac{\xi}{n}\Big) \sim 16 \gamma^2 \Big[\lambda\gamma 
    + 2\pi^2 \left(\xi^2 + (\eta+\xi)^2\right)\Big].
\end{equation}}
Since the block entries of $\mb M_n$ commute we can also write
\begin{equation*}
 \mb M_n^{-1} =
 \begin{bmatrix}
\Big[A _nB_n-(\ga^+_n\ga^-_n n^4){\rm Id}_2\Big]^{-1} &0\\
0&\Big[A _nB_n-(\ga^+_n\ga^-_n n^4){\rm Id}_2\Big]^{-1}
 \end{bmatrix} \begin{bmatrix}
B_n \vphantom{\Big]^{-1}} &\ga^-_n n^2\;{\rm Id}_2\\
\ga^+_n n^2\;{\rm Id}_2&A _n\vphantom{\Big]^{-1}}
 \end{bmatrix}.
 \end{equation*}
Note that
$$
\Big[A _nB_n-(\ga^+_n\ga^-_n n^4){\rm Id}_2\Big]^{-1}=\frac{1}{\mathrm{det}(\mb M_n)}\begin{bmatrix}
a_n^\star b_n+n^3\;(\delta\ga_n) &2\ga^- n^2\;{\rm Re}[
a_n]\\
&\\
2\ga^+ n^2\;{\rm Re}[b_n]&a_n b_n^\star+n^3\;(\delta\ga_n) 
 \end{bmatrix}.
$$
With these formulas we conclude that
\begin{equation}
\label{022308}
\mb M_n^{-1} =\frac{1}{\text{det}(\mb M_n)} \begin{bmatrix}
\displaystyle \vphantom{\int} d_n^+ & \ga^-d_n & \ga_n^-c_n^\star &\gamma^-\gamma^-_n  c_n^0 \\
\displaystyle \vphantom{\int}\ga^+d_n^0 & d_n^- & \gamma^+\gamma^-_nc_{n}^0 & \ga_n^-c_n \\
\displaystyle \vphantom{\int}\ga^+_nc_n^\star & \gamma^-\gamma^+_n c_{n}^0 & (d_{n}^-)^\star & \ga^-(d_n^0)^\star \\
\displaystyle \vphantom{\int}\gamma^+\gamma^+_n c_n^0 &\ga^+_n c_n & \ga^+d_n^\star & (d_n^+)^\star
\end{bmatrix}
\end{equation}
where all the constants are explicit and given by
\begin{equation}
\label{041203}
\left\{
\begin{aligned}
d_n^+ & := \; a_n^\star \big|b_n\big|^2+n^3b_n^\star(\delta\ga_n) - 2\gamma^-\ga^+ n^4\; \text{Re}[a_n], \vphantom{\Big(}
\\
d_n^- & := \; b_n^\star \big|a_n\big|^2 +n^3a_n^\star (\delta\ga_n)- 2\gamma^-\ga^+ n^4\;
\text{Re}[b_n], \vphantom{\Big(}\\
d_{n} & := \; 2 n^2 a_n^\star\; {\rm Re}[a_n] - n^2
a_n^\star b_n-n^5 (\delta\ga_n) ,
\vphantom{\Big(}\\
d_n^0 & := \; 2n^2 b_n^\star\; {\rm Re}[b_n] - n^2
a_n b_n^\star-n^5 (\delta\ga_n) , \vphantom{\Big(}\\
c_n & :=  \; n^{2} a_n b_n^\star+n^5 (\delta\ga_n)\vphantom{\Big(}, \\ 
c_n^0 & :=  \; 2n^{4} \;\text{Re}[a_n].
\end{aligned}\right.
\end{equation}

\subsection{Asymptotics of the coefficients} \label{ssec:asymp}
Substituting from \eqref{a-n} into the respective formulas of \eqref{041203} 
and then identifying  the order of magnitude of
the appearing terms we conclude the following:
\begin{lemma} \label{lem:estimates} 
The following asymptotic equalities hold:
\begin{eqnarray}
&&
\frac{d_n^+}{n^6}=4\ga^3+4\ga \sin^2(2\pi k)+2\ga (\si_n
  s)^2- \frac{i(\delta_n s)}{n}\; \big(4\ga^2+(\si_ns)^2\big)+\frac{C_n}{n^2},\nonumber \vphantom{\Bigg\{}\\
&&
\frac{c_n}{n^6}=4\ga^2-2i\ga (\si_n s)+\sin^2\big(2\pi
    \big(k+\tfrac{\eta}{n}\big)\big)-\sin^2(2\pi
    k)+\frac{(\delta_n s)(\si_n s)}{n}+\frac{C_n}{n^2},\nonumber \vphantom{\Bigg\{} \\
&&
\frac{d_n}{n^6}=4\ga^2-2i\ga (\si_n s)+\sin^2(2\pi
    k)-\sin^2\big(2\pi
    \big(k+\tfrac{\eta}{n}\big)\big)-2i\ga \frac{(\delta_n s)}{n}-\frac{(\delta_n s)(\si_n s)}{n}+\frac{C_n}{n^2},\nonumber\vphantom{\Bigg\{} \\
&&
\frac{c_n^0}{n^6}=4\ga+\frac{C_n}{n^2}, \nonumber\vphantom{\Bigg\{}\\
&&
\frac{d_n^0}{n^6}=4\ga^2-2i\ga (\si_n s) +\sin^2(2\pi
    k)-\sin^2\big(2\pi
    \big(k+\tfrac{\eta}{n}\big)\big) -2i\ga \frac{(\delta_n s)}{n}-\frac{(\delta_n s)(\si_n s)}{n}+\frac{C_n}{n^2},\nonumber\vphantom{\Bigg\{}\\
&&
\frac{d_n^-}{n^6}=4\ga^3+4\ga\sin^2(2\pi k)-4i\ga^2 (\si_n s)+2\ga\Big(\sin^2\big(2\pi
    \big(k+\tfrac{\eta}{n}\big)\big)-\sin^2(2\pi
    k)\Big)+\frac{C_n}{n^2},\nonumber
\end{eqnarray}
with $C_n\preceq 1$.
\end{lemma}

 From the above asymptotics it is clear that 
$
{\bf M}^{-1}_n(\lambda,\eta,k) \to 0$ for a fixed $k\neq 0$.
 In addition,
 ${\bf M}^{-1}_n(\lambda,\eta,\frac{\xi}{n})$ and $n^2 \; {\rm e}^{\rm T}\; {\bf M}^{-1}_n(\lambda,\eta,k)$ 
tend to finite limits that we need to compute explicitly in order to complete the proof.

Using \eqref{022308}, \eqref{032308} and the formulas for the asymptotics of  the entries  of ${\bf M}^{-1}_n$, provided by Lemma \ref{lem:estimates}, we conclude that
\begin{equation}
  \label{eq:macrolim}
  \lim_{n\to\infty} {\bf M}^{-1}_n\Big(\lambda,\eta,\frac{\xi}{n}\Big) = 
\frac{\gamma}{4\lambda\gamma + 8{\pi^2}\left[\xi^2 + (\xi + \eta)^2\right]} \,{\bf 1}\otimes{\bf 1}.
\end{equation}
The above in particular implies that ${\rm e}^{\rm T} \cdot  \mb M_n^{-1} (\la,\eta,\frac{\xi}{n})  {\bf 1}\to0$, for any $\xi\in\Z$.
By the same token we can also compute
 \begin{equation}
\lim_{n\to+\infty}\Big\{ n^2\; {\rm e}^{\rm T} \cdot  \mb M_n^{-1} \Big(\la,\eta,\frac{\xi}{n}\Big)  {\bf 1} \Big\}=
\frac{4 \pi^2\; \xi(\xi+\eta)}{\la\ga^2+2\ga\pi^2\;\left[\xi^2 + (\xi + \eta)^2\right]}, \label{eq:macrolim2}
\end{equation}
and
\begin{equation}
  \label{eq:macrolimpr}
   \lim_{n\to\infty} \Big\{n^2 \; {\rm e}^{\mathrm{T}}\cdot\; {\bf M}^{-1}_n(\lambda,\eta, k)\Big\} = \frac 1{2\gamma} \left[1,0,0,1\right],\quad k\neq 0.
\end{equation}
We prove here only \eqref{eq:macrolim2} and we let the reader verify \eqref{eq:macrolim} and \eqref{eq:macrolimpr} using similar computations.  An explicit calculation gives
 \begin{equation}
n^2 \; {\rm e}^{\rm T} \cdot ( \mb M_n^{-1} \; {\bf 1} )= \frac{n^2\; \Xi_n}{\text{det}(\mb M_n)}, \label{eq:explicit}\end{equation}
where
\begin{align*}
\Xi_n (\la,\eta,k):=&\;2{\rm Re}\big[d_n^+-d_n^-
\big]+\ga^-(d_n-(d_n^0)^\star)+\ga^+(d_n^\star-d_n^0) 
\\ &+4i\sin\big(2\pi
\big(k+\tfrac{\eta}{n}\big)\big){\rm Im}[c_n]+4\sin(2\pi k)\sin\big(2\pi
\big(k+\tfrac{\eta}{n}\big)\big)c_n^0. \vphantom{\bigg\{}
\end{align*}
Substituting from \eqref{041203} yields:
\begin{align}
\Xi_n (\la,\eta,k)
= & (\la+2\ga n^2)\bigg\{2\big[n^4(\si_ns)^2
-n^2(\delta_ns)^2\big]
\notag \\ 
&+4i n^4 \Big(\sin(2\pi k)-\sin\big(2\pi \big(k+\tfrac \eta n\big)\big)\Big)(\si_n s)
\nonumber\\
&+4in^3\Big(\sin\big(2\pi \big(k+\tfrac \eta n\big)\big) +\sin(2\pi k)\Big) (\delta_n s) \notag \\
&+8n^{4}\sin(2\pi k)\sin\big(2\pi \big(k+\tfrac \eta n\big)\big)\bigg\}.\label{012004}
\end{align}
From the above and a direct calculation we obtain \eqref{eq:macrolim2}.

\subsection{Proof of Proposition \ref{lem:mech} }
\label{sec:proof-proposition-mech}

Basically the argument follows
 the same idea as the proof of  Proposition \ref{propo021004}. In fact \eqref{eq:macrolim} implies that 
$\overline{\mathfrak w}_n(\lambda,\eta, k)$ concentrates on {small} $k$-s like $\overline{W}^+_n(\eta,k)$. The main difficulty is to deal with the averaging $[\cdot]_n$. For that purpose, for  $\rho\in \left(0,\frac 12\right)$, define
 \begin{equation}
\label{032004}
\widehat\T_{n,\rho}:=\left\{k\in \widehat\T_n:\,|\sin(\pi k)|\ge n^{-\rho}\right\}
\end{equation}
and its complement
$\widehat\T^c_{n,\rho}:=\widehat{\T}_n\setminus\widehat\T_{n,\rho}$. Recall the left hand side of \eqref{011207}: it can be written as ${\rm I}_n+{\rm II}_n$,
where ${\rm I}_n$ and ${\rm II}_n$ correspond to the summations over
$\widehat\T_{n,\rho}$ and  $\widehat\T_{n,\rho}^c$ respectively.
{First we show that for $G\in \mc P_M$
\begin{equation}
  \label{eq:mech2}
 {\rm I}_n= \sum_{\eta\in\bb Z} \frac 1n \sum_{k\in \widehat{\T}_{n,\rho}} \left({\bf M}_n^{-1}(\lambda,\eta,k) {\bf 1}\right) \overline{W}^+_n(\eta,k)
 ({\mc F} G)^\star(\eta,k) \  \mathop{\longrightarrow}_{n\to\infty}  0.
\end{equation}
In fact we have
\begin{multline*}
    \bigg\| \sum_{\eta\in\bb Z} \frac 1n \sum_{|\xi|\ge n^\rho} \Big({\bf M}_n^{-1}\Big(\lambda,\eta,\frac{\xi}n\Big) {\bf 1}\Big)
 \overline{W}^+_n\Big(\eta,\frac{\xi}n\Big)
 ({\mc F} G)^\star\Big(\eta,\frac{\xi}n\Big)\bigg\|_\infty\\
  \le C\|G\|_0  \sum_{|\eta| \le M} \sum_{|\xi| \ge n^\rho} \Big\|{\bf M}_n^{-1}\Big(\lambda,\eta,\frac{\xi}n\Big) {\bf 1}\Big\|_\infty  
\  \mathop{\longrightarrow}_{n\to\infty}  0,
  \end{multline*}
where the constant $C$ depends only on the initial mechanical energy.
Then by the same argument as the one used in the proof of Proposition \ref{propo021004}, and from \eqref{050704b} and \eqref{eq:macrolim} we have
\begin{multline*}
    \lim_{n\to\infty} \sum_{|\eta|\le M} \frac 1n \sum_{|\xi|< n^\rho}
    \Big({\bf M}_n^{-1}\Big(\lambda,\eta,\frac{\xi}n\Big) {\bf 1}\Big)
    \overline{W}^+_n\Big(\eta,\frac{\xi}n\Big) ({\mc F} G)^\star\Big(\eta,\frac{\xi}n\Big)\\
   = \sum_{\eta,\xi \in \bb Z} \frac{\gamma W(r_0\; ;\; \eta,\xi)}{\lambda \gamma + 2 \pi^2\left[\xi^2 + (\xi + \eta)^2\right]} 
    ({\mc F} G)^\star(\eta,0) \; {\bf 1}.
  \end{multline*}}
  This concludes the proof of \eqref{011207}.
Concerning the proof of \eqref{eq:nonlin1}, recall that \[\overline{\cal I}_n = \Big[{\rm e} \cdot (\mb M_n^{-1} \; {\bf 1})\;  \overline{W}^+_n\Big]_n,\] and that we have already computed the limit \eqref{eq:macrolim2}.
Consequently, the result will follow if we are able to show that the contribution to the $k$-averaging from the higher frequencies is negligible. 
The quantity $n^2\; \overline{\mc I}_n$ can be written as ${\rm I}_n+{\rm II}_n$,
where ${\rm I}_n$ and ${\rm II}_n$ correspond to the summations over
$\widehat\T_{n,\rho}$ and  $\widehat\T_{n,\rho}^c$ respectively.
Using the explicit computations \eqref{eq:explicit}, \eqref{012004} and \eqref{051403} we can write
\begin{align*}
n^2\; {\rm e} \cdot  (\mb M_n^{-1}\; {\bf 1}) - {\ga}^{-1}= &\;  \ga^{-1} \left\{2\Big(\sin(2\pi k)-\sin\big(2\pi \big(k+\tfrac \eta n\big)\big)\Big)^2+\frac{C_n}{n}\right\}\\
& \times 
\left\{\vphantom{\int_0^1}2 \sin^2(2\pi
    k) +2\sin^2\big(2\pi \big(k+\tfrac \eta n\big)\big)+(\sigma_n s)^2 +\frac{C_n'}{n^2}\right\}^{-1}.
\end{align*}
It is clear from the above equality that
\begin{equation}
\label{012004a}
\lim_{n\to+\infty}\sup_{|\eta|\le M}\sup_{k\in\widehat\T_{n,\rho}}\left|n^2 \; {\rm e} \cdot  (\mb M_n^{-1} \;{\bf 1}) (\la,\eta,k) - \frac{1}{\ga}\right|=0.
\end{equation}
Thanks to \eqref{012004a} we conclude that
$\lim_{n\to+\infty}({\rm I}_n-{\rm I}_n')=0$, where
\begin{equation}
\label{032104}
{\rm I}_n':=\frac{1}{\ga n}\sum_{k\in \widehat\T_{n,\rho}}\overline{ W}_n^+(\eta,k).
\end{equation}
After a straightforward calculation using the definition of
$\overline{ W}_n^+(\eta,k)$ (see \eqref{eq:Wbar} and \eqref{eq:equality})  we conclude that
\eqref{032104} equals
$$
\frac{1}{2\ga }\sum_{\xi,\xi'\in\Z}\sum_{k\in
  \widehat \T_{n,\rho}}1_{\Z}\Big(k-\frac{\xi'}{n}\Big) 1_{\Z}\Big(-k+
  \frac{\xi-\eta}{n}\Big) (\mc F r_0)(\xi)  (\mc F  r_0)^\star(\xi'),
$$
where $1_{\Z}$ is the indicator function of the integer lattice. Due
to the assumed separation of $k$ from $0$, see \eqref{032004}, and the
decay of the Fourier coefficients of $r_0(\cdot)$ (that belongs to ${\cal C}^\infty(\T)$) we
conclude from the above that $\lim_{n\to+\infty}{\rm I}_n'=0$, thus also $\lim_{n\to+\infty}{\rm I}_n=0$.

Moreover, a similar calculation also yields
\begin{equation*}
{\rm II}_n
=\frac{1}{2}\sum_{\xi\in \mc N_{\rho,n}}n^2\; {\rm e} \cdot ( \mb M_n^{-1} \; {\bf 1} )\Big(\la,\eta,\frac{\xi}{n}\Big) (\mc F r_0)(\xi+\eta)  (\mc F r_0)^\star(\xi),
\end{equation*}
where
$$
\mc N_{\rho,n}:=\Big\{\xi\in\Z\; :\; |\xi|\le \tfrac n 2,\quad\big|\sin\big(
    \tfrac{\pi  \xi}{n}\big)\big|\le n^{-\rho}\Big\}.
$$
Using the dominated convergence theorem we conclude from \eqref{eq:macrolim2} that
\begin{equation}
\label{043104}
\lim_{n\to+\infty}{\rm II}_n
=\sum_{\xi\in\Z}\frac{4\pi^2 \xi(\xi+\eta)\; W(r_0\; ; \; \eta,\xi)}{\la\ga^2+2\ga\pi^2\;[\xi^2+(\xi+\eta)^2]} .
\end{equation}

\subsection{Proof of Lemma \ref{lemma2} and Proposition \ref{lemma3} }
\label{sec:proof-lemma2}

Since \eqref{eq:1st} is a direct consequence of \eqref{eq:2nd}, we prove directly \eqref{eq:2nd}, that is a 
consequence of the following lemma. 

\begin{lemma} \label{lem:estimateS} 
The following asymptotic equality holds:
\begin{equation}
\label{041403}
{\cal S}_n(\la,\eta) :=n^2\;\big(1-\gamma n^2\; \mc M_n(\lambda,\eta)\big)
=\frac{1}{2\ga}\Big(\la+\frac{\eta^2\pi^2}{\ga}\Big)+\frac{C_n}{n},
\end{equation}
where $|C_n|\preceq 1$.
 \end{lemma}

\begin{proof}
After a direct calculation, we obtain
\[\mc S_n(\lambda,\eta) 
=\left[\frac{n^2}{\text{det}(\mb M_n)} \;
\big({\text{det}(\mb M_n)}-\ga n^2\Theta_n\big)\right]_n, \]
where 
\[\Theta_n :=2{\rm Re}\Big[d_n^++d_n^-+2\ga^2 c_n^0-2\ga  c_n \Big]-\ga^-\big(d_n+(d_n^0)^\star\big)-\ga^+\big(d_n^\star+d_n^0\big).\]
Therefore, from \eqref{041203} and Lemma \ref{lem:estimates}, we have
\[
 {\cal S}_n(\la,\eta) =\frac{\la}{2\ga n}\sum_{k\in\widehat{\T}_n}{\rm I}_n(\eta,k)
 + \frac{1}{4\ga^2 n}\sum_{k\in\widehat{\T}_n}{\rm II}_n(\eta,k).
\]
where
\begin{eqnarray*}
&&
{\rm I}_n:=\frac{\Gamma_n
  +2\la\ga+(\delta_n s)^2+\frac{C_n}n}{\Gamma_n
  +4\la\ga+(\delta_n s)^2+\ga^{-2}\big((\delta_n
  s)(\si_ns)+(\delta\ga_n)\big)^2+\frac{C_n'}n},\\
&&
{\rm II}_n:=
\frac{n^2\; \big((\si_n s)
   (\delta_n s)+(\delta\ga_n)\big)^2}{\Gamma_n
  +4\la\ga+(\delta_n s)^2+\ga^{-2}\big((\delta_n
  s)(\si_ns)+(\delta\ga_n)\big)^2+\frac{C_n'}n}.
\end{eqnarray*}
The expressions $C_n,C_n'$ satisfy $|C_n|+|C_n'
|\preceq 1$.
Directly from the definition of ${\rm I}_n$ and ${\rm II}_n$ we conclude that $|{\rm I}_n|+|{\rm II}_n|\preceq 1$ and
\begin{eqnarray*}
&&\lim_{n\to+\infty}\frac{1}{n}\sum_{k\in\widehat{\T}_n}{\rm I}_n=1,\\
&&
\lim_{n\to+\infty}\frac{1}{n}\sum_{k\in\widehat{\T}_n}{\rm II}_n=(2\pi\eta)^2\int_{\T}\frac{\big[4\sin(2\pi v)\sin^2(\pi
  v)+\sin(4\pi v)\big]^2}{4\sin^2(2\pi
  v)+16\sin^4(\pi v)}\dd v.
\end{eqnarray*}
Using trigonometric identities \[ 2\sin^2(\pi v)=1-\cos(2\pi v) \quad \text{ and } \quad \sin(4\pi v)=2\sin(2\pi v)\cos(2\pi v)\] we
conclude
that the last integral equals
$$
\int_{\T}\frac{\big[  2\sin(2\pi v)(1-\cos(2\pi v))+\sin(4\pi v)\big]^2}{4\sin^2(2\pi
  v)+4[1-\cos(2\pi v)]^2}\dd v
=\int_{\T}\frac{\sin^2(2\pi v)}{2(1-\cos(2\pi
  v))}\dd v=\frac{1}{2}.
$$
Thus, we obtain
$$
{\cal S}_n(\lambda,\eta)  =\frac{1}{2\ga}\Big(\la+\frac{\pi^2\eta^2}{\ga }\Big)+\frac{C_n}{n},
$$
with
 $|C_n|\preceq 1$.
\end{proof} 

It remains to prove \eqref{eq:4th}.
This would be a direct consequence of \eqref{eq:macrolimpr}, 
 but we need some care in exchanging the limit with the $[\cdot]_n$ averaging.

Choose $\rho\in(0,1)$, then we can decompose
\begin{equation*}
    n^2 \Big[ {\rm e}\cdot \big({\bf M}^{-1}_n\; \widetilde{\mf v}^0_n\big) \Big]_n = 
    \frac {1}{2\gamma} \left(\Big[\widetilde{W}^+_n\Big]_n + \Big[\widetilde{W}^-_n\Big]_n\right)
    + K^{(1)}_n +  K^{(2)}_n
\end{equation*}
where
\begin{equation*}
   K^{(1)}_n (\lambda, \eta) = \frac 1n \sum_{k\in\widehat{\T}_{n,\rho}} 
   \Big(n^2{\rm e}^{\rm T}\; {\bf M}^{-1}_n(\lambda,\eta,k) - \frac 1{2\gamma} \; {\rm u} \Big) \cdot
   \widetilde{\mf v}^0_n (\eta,k)
\end{equation*}
with ${\rm u} = [1,0,0,1]^{\rm T}$, and the definition of $ K^{(2)}_n$ differs from  $K^{(1)}_n$ only in that the range of the summation
in $k$ extends over $\widehat{\T}_{n,\rho}^c$.

From \eqref{eq:convergence_tilde} we have
\begin{equation*}
  \lim_{n\to\infty} \left[\widetilde{W}^+_n(\eta,\cdot)\right]_n = \lim_{n\to\infty} \left[\widetilde{W}^-_n(\eta,\cdot)\right]_n = 
  \big({\mc F} e_{\rm thm}(0,\cdot) \big) (\eta), 
\end{equation*}
and therefore we only have to prove that $K^{(1)}_n$ and $K^{(2)}_n$ vanish as $n\to\infty$. 
Concerning $K^{(1)}_n$ we write
\begin{align}
    \left| K^{(1)}_n(\lambda,\eta) \right|  & \le \sup_{k \in \widehat{\T}_{n,\rho} } 
    \left\|n^2 {\rm e}^{\rm T} \;{\bf  M}^{-1}_n(\lambda,\eta,k) - \frac 1{2\gamma} \; {\rm u} \right\|_\infty
    \ \Big\| \left[ \widetilde{\mf v}^0_n (\eta,\cdot)\right]_n
    \Big\|_\infty \notag\\
  &  \le C \sup_{k \in  \widehat{\T}_{n,\rho} } \left\| n^2{\rm e}^{\rm T}\; {\bf M}^{-1}_n(\lambda,\eta,k) - \frac 1{2\gamma} \; {\rm u} \right\|_\infty \label{eq:est1}  , 
  \end{align}
where $C$ depends on the bound on the initial energy. By  direct estimation, using the information on the asymptotic behavior for 
the coefficients of  ${\bf M}^{-1}_n$, provided by \eqref{eq:macrolimpr},  we conclude  that the right hand side of \eqref{eq:est1} converges to $0$ as $n\to\infty$, for 
any given $\lambda$ and $\eta$. 

Concerning $K^{(2)}_n$, since $n^2\; {\rm e}^{\rm T}\cdot\;{\bf M}^{-1}_n(\lambda, \eta, k)$ are uniformly bounded in $k$, for any integer 
 $M$ there exists a constant $C_M>0$ such that, for all $n$, $\lambda >\lambda_M, |\eta| \le M$,
\begin{equation*}
    \left| K^{(2)}_n(\lambda,\eta) \right| \le \frac{C_M}{n} \sum_{k\in\widehat{\T}^c_{n,\rho}} \left|  \widetilde{\mf v}^0_n (\eta,k) \right|.
\end{equation*}
Using the Cauchy-Schwarz inequality we have
\begin{equation*}
  \begin{split}
    \left| K^{(2)}_n(\lambda,\eta) \right| &  \le \frac{C'_M}{n}\;
    \big|\widehat{\T}^c_{n,\rho}\big|^{\frac 12}\; \Bigg(\sum_{k\in
        \widehat{\T}_{n}} \mathbb E \left|  \widetilde{\mf v}^0_n (\eta,k) \right|^2
    \Bigg)^{\frac12}\\
  &  \le {C'_M} \;\bigg( \frac{ \big|\widehat{\T}^c_{n,\rho}\big|}{n}\bigg)^{\frac12} \; w_*^{\frac12} \mathop{\longrightarrow}_{n\to\infty} 0,
  \end{split}
\end{equation*}
where $w_*$ is given by \eqref{071407}. This concludes the proof of Proposition \ref{lemma2}.

 Proposition \ref{lemma3} is also a direct consequence of  \eqref{eq:macrolimpr}: {instead of computing the limit of 
 \[ \widetilde{\mc I}_n(\lambda,\eta) = \Big[\big(\widetilde{w}_n^+ - \widetilde{y}_n^+ - \widetilde{y}_n^- + \widetilde{w}_n^-\big)(\lambda,\eta,\cdot)\Big]_n\] we can compute the limit of 
 \[ 2 \; \big[ \widetilde{w}_n^+(\lambda,\eta,\cdot)\big]_n\] by using very similar arguments.}

\section*{Acknowledgments}
T. K. acknowledges the support of the
  Polish National Science Center grant DEC-2012/07/B/ST1/03320.
  
The work of M.S. was  supported by the 
ANR-14-CE25-0011 project (EDNHS) of the French National Research Agency (ANR), and  by the Labex CEMPI (ANR-11-LABX-0007-01), and  finally by  CAPES (Brazil) and IMPA (Instituto de Matematica Pura e Aplicada, Rio de Janeiro) through a post-doctoral fellowship.

S.O. has been partially supported by the ANR-15-CE40-0020-01 grant LSD.

\bibliographystyle{amsplain}

\end{document}